\documentclass[reqno,12pt]{amsart}
\usepackage[margin=1in,footskip=0.5in]{geometry}

\usepackage{enumitem}
\newcommand{\FF}{F_0}
\usepackage[latin1]{inputenc}
\usepackage[english]{babel}
\usepackage{amsaddr}
\usepackage{subfigure}
\usepackage{amsmath,amssymb,amsthm}
\usepackage{ae}
\usepackage[T1]{fontenc}
\usepackage{graphicx}
\usepackage{epstopdf}
\usepackage{color}

\newcommand{\e}{{\mathrm e}}

\newcommand{\be}{\begin{eqnarray}}
\newcommand{\ee}{\end{eqnarray}}
\newcommand{\ba}{\begin{align}}
\newcommand{\ea}{\end{align}}
\newcommand{\bi}{\begin{itemize}}
\newcommand{\ei}{\end{itemize}}
\usepackage{fullpage}

\newcommand{\secref}[1]{Section~\ref{sec:#1}}

\newcommand{\seclab}[1]{\label{sec:#1}}
\newcommand{\eqlab}[1]{\label{eq:#1}}
\renewcommand{\eqref}[1]{(\ref{eq:#1})}

\newcommand{\figref}[1]{Fig.~\ref{fig:#1}}
\newcommand{\figlab}[1]{\label{fig:#1}}
\newcommand{\propref}[1]{Proposition~\ref{proposition:#1}}
\newcommand{\proplab}[1]{\label{proposition:#1}}
\newcommand{\lemmaref}[1]{Lemma~\ref{lemma:#1}}
\newcommand{\lemmalab}[1]{\label{lemma:#1}}

\newcommand{\remlab}[1]{\label{remark:#1}}
\newcommand{\thmref}[1]{Theorem~\ref{theorem:#1}}
\newcommand{\thmlab}[1]{\label{theorem:#1}}

\newcommand{\appref}[1]{Appendix~\ref{app:#1}}

\newcommand{\applab}[1]{\label{app:#1}}

\newtheorem{theorem}{Theorem}[section]
\newtheorem{proposition}[theorem]{Proposition}

\newtheorem{lemma}[theorem]{Lemma}

\newtheorem{remark}[theorem]{Remark}

\numberwithin{equation}{section}
\newcommand\rspp[1]{{\color{black}{#1}}}

\title{A geometric approach to exponentially small splitting: The generic zero-Hopf bifurcation of co-dimension two}
\author{K. Uldall Kristiansen}
\address{Department of Applied Mathematics and Computer Science,
Technical University of Denmark,
2800 Kgs. Lyngby,
Denmark }
\begin{document}
 \begin{abstract}
 In this paper, we consider the unfolding of the real-analytic and generic zero-Hopf bifurcation of co-dimension two. It is well-known that in an open set of parameter space the splitting of one-dimensional stable and unstable manifolds is beyond all orders. This paper provides a new geometric dynamical-systems-oriented proof for the exponentially small splitting. As a novel aspect, we relate the exponentially small splitting to the lack of analyticity of center-like manifolds of generalized saddle-nodes. Moreover, the blowup method plays an important technical role as a systematic way to relate dynamics on different orders of magnitude.  Finally, as our approach takes place in the (complexified) phase space, we do not rely on an explicit time-parametrization of the unperturbed heteroclinic connection. We therefore believe that our approach has general interest.



 \end{abstract}
 \bigskip
\smallskip

\maketitle

\tableofcontents
\section{Introduction}
In this paper, we revisit the exponentially small splitting in the unfolding of the real-analytic zero-Hopf bifurcation. We recall that the zero-Hopf bifurcation \cite{Guckenheimer97} is defined for a vector-field $\Xi_0$ in $\mathbb R^3$ as a singularity with eigenvalues  $0$, $\pm i\omega$, $\omega \ne 0$, of the linearization. We will consider the generic case where the zero-Hopf bifurcation is co-dimension two \cite{MR779710}, using $(\mu,\nu)\in \mathbb R^2$, $(\mu,\nu)\sim (0,0)$, to denote the unfolding parameters. In particular, by \cite[Eq. (6)]{baldom2013a} we have the following normal form for the unfolding $\Xi_{\mu,\nu}$ of $\Xi_0$:
\begin{equation}\eqlab{model0}
\begin{aligned}
 x' &=x^2-\mu+a (y^2 +z^2) + F(x,y,z,\mu,\nu),\\
 y'&= (\nu-b x ) y +  z+G(x,y,z,\mu,\nu),\\
 z'&= (\nu-b x) z -  y+H(x,y,z,\mu,\nu),
\end{aligned}
\end{equation}
with $(\cdot )' = \frac{d}{d\tau}$.
Here $F,G$ and $H$ are real-analytic functions, defined in a neighborhood of $(x,y,z,\mu,\nu)=(0,0,0,0,0)$ in $\mathbb C^5$, and each function is third order:
\begin{align}\eqlab{Wthird}
 W(x,y,z,\mu,\nu)  = \mathcal O(\vert (x,y,z,\mu,\nu)\vert^3),\quad W=F,G,H,
\end{align}
with respect to $(x,y,z,\mu,\nu)\to (0,0,0,0,0)$. 
For the derivation of \eqref{model0}, we have used a re-parametrization of time that brings $\omega=1$, see \appref{model} for details.
Here we also show that it is without loss of generality to assume that
\begin{align}
 F(x,0,0,\mu,\nu) = \FF x^3+\mathcal O(\vert (x,\mu,\nu)\vert^4), \quad \FF\in \mathbb R,\eqlab{Fexpansion}
\end{align}
with respect to $(x,\mu,\nu)\to (0,0,0)$.
%
%
%
We will as in \cite{baldom2013a} suppose that
\begin{align*}
 b>0.
\end{align*}
Then we have the following:  Fix any $\delta\in (0,1)$ and let $\mathcal W$ denote the region of parameter space $(\mu,\nu)\in \mathbb R^2$, defined by
\begin{align}\eqlab{mathcalW}
 \begin{cases}
  \mu = \epsilon^2,\\
  \nu = \epsilon b \sigma,
 \end{cases}\quad \sigma \in [-1+\delta,1-\delta],\quad 0<\epsilon<\epsilon_0,
\end{align}
with $0<\epsilon_0=\epsilon_0(\delta)\ll 1$, see \figref{mathcalW}.
Then for any $(\mu,\nu)\in \mathcal W$, the system \eqref{model0} has
two saddle-foci $E^{\pm} (\mu,\nu)$  with $x$-components given by the form $\pm \epsilon (1+\mathcal O(\epsilon))$, respectively. For further details, see \cite{baldom2013a} or \secref{unperturbed} below. There is a one-dimensional heteroclinic connection for $\epsilon=0$, as well as a formal one in powers of $\epsilon$. The splitting of the one-dimensional unstable and stable manifolds of $E^{\pm}$ is therefore a beyond-all-orders phenomena with respect to $\epsilon\rightarrow 0$. We are as in \cite{baldom2013a} concerned with an expression for their splitting for all $0<\epsilon\ll 1$.

 \begin{figure}[h!]
\begin{center}
{\includegraphics[width=.5\textwidth]{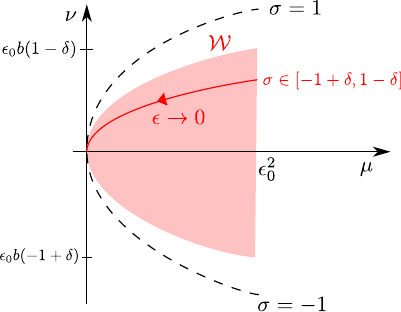}}
\end{center}
\caption{Illustration of the region $\mathcal W\subset \mathbb R^2$ in the parameter $(\mu,\nu)$-plane. For any $(\mu,\nu)\in \mathcal W$, the system \eqref{model0} has
two saddle-foci $E^{\pm} (\mu,\nu)$. The splitting of the associated unstable and stable manifolds is beyond all orders as $\epsilon\to 0$ (with $\sigma\in (-1,1)$ fixed), see \eqref{mathcalW}.}
\figlab{mathcalW}
\end{figure}
\begin{remark}
For $a>0$, there is a related splitting problem associated with the remaining two-dimensional invariant manifolds (carrying focus-type dynamics) of $E^{\pm}(\mu,\nu)$, which is important for the existence of Shilnikov bifurcations, see \cite{baldom2019a} as well as \cite{MR779710} for Shilnikov bifurcations in the $C^\infty$-context. However, in this paper, we will only focus on the one-dimensional splitting problem, and will therefore not make any assumptions on the parameter $a$. For details on $(\mu,\nu,b)\notin \mathcal W\times \mathbb R_+$, we refer to \cite{MR779710,Guckenheimer97}.
\end{remark}

We now describe the result of \cite{baldom2013a} in further details. For this, consider
\eqref{model0} with $b>0$ and $(\mu,\nu)\in \mathcal W$ given by \eqref{mathcalW}:
\begin{equation}\eqlab{model}
\begin{aligned}
 x' &=x^2-\epsilon^2+a (y^2 +z^2) + F(x,y,z,\epsilon),\\
 y'&= b(\sigma \epsilon -x ) y +  z+G(x,y,z,\epsilon),\\
 z'&= b(\sigma \epsilon- x) z -  y+H(x,y,z,\epsilon),
\end{aligned}
\end{equation}
where we from now on use $W(x,y,z,\epsilon)$ to denote $W(x,y,z,\epsilon^2 ,\epsilon b\sigma)$ for $W=F,G,H$. By assumption, we have
\begin{align}
 W(x,y,z,\epsilon) =  \mathcal O(\vert (x,y,z,\epsilon)\vert^3),\quad W=F,G,H.\eqlab{R0}
\end{align}
Now, with  $b>0$ and $\sigma \in (-1,1)$ fixed, the main result of \cite{baldom2013a} can then be summarized as an asymptotic expression for the distance $d(\epsilon)$, $0<\epsilon\ll 1$, between the one-dimensional unstable and stable manifolds of $E^{\pm}(\epsilon):=E^{\pm}(\epsilon^2,\epsilon b\sigma )$, respectively, at the fixed section $\{x=0\}$:
\begin{align}
d(\epsilon) = \epsilon^{-b} \mathrm{e}^{-\frac{\pi}{2}\epsilon^{-1}}\left(C_0+\mathcal O(1/\log \epsilon^{-1})\right),\eqlab{dist}
\end{align}
as $\epsilon\to 0$, for some constant $C_0\ge 0$ independent of $\epsilon$ and $\sigma$. Notice when comparing with \cite{baldom2013a} that our normal form corresponds to $\alpha_0=1$, $\alpha_1=\alpha_3=0$, $h_0=\FF$ and that our $C_0$ is therefore related to their $C^*$ by
\begin{align*}
C_0 = \e^{\frac{\pi}{2} \FF}C^*.
\end{align*}
Generically, $C_0\ne 0$ and \eqref{dist} therefore determines an asymptotic formula for the leading order splitting.

The proof of the expression \eqref{dist} in \cite{baldom2013a} follows a well-known approach (initiated by Lazutkin for maps, see \cite{lazutkin2005a}):
First, the authors scale the variables. \rspp{In the present paper, we will use the following scaling
\begin{align}\label{eq:buhopf}
\begin{cases} x=\epsilon x_2, \\ y=\epsilon^2 y_2, \\ z=\epsilon^2 z_2,
\end{cases}
 \end{align}
 rather than the scaling used in \cite{baldom2013a}
 \begin{align}\nonumber
\begin{cases} x=\epsilon x_2, \\ y=\epsilon y_2, \\ z=\epsilon z_2,
\end{cases}
 \end{align}
 but this difference is not essential. The transformation defined by} \eqref{buhopf}  brings \eqref{model} into the following slow-fast form
 \begin{equation}\label{eq:zeroHopf}
\begin{aligned}
 \dot x_2 &=x_2^2-1+\epsilon^2 a (y_2^2+z_2^2)+\epsilon F_2(x_2,\epsilon y_2,\epsilon z_2,\epsilon),\\
 \epsilon \dot y_2&= z_2+\epsilon b(-x_2+\sigma) y_2+\epsilon G_2(x_2,\epsilon y_2,\epsilon z_2,\epsilon),\\
 \epsilon \dot z_2&= -y_2+\epsilon b(-x_2+ \sigma) z_2 +\epsilon H_2(x_2,\epsilon y_2,\epsilon z_2,\epsilon),
\end{aligned}
\end{equation}
where $\frac{d}{dt}=\dot{(\cdot)}=\epsilon^{-1} (\cdot )'$
and
\begin{align*}
 W_2(x_2,\epsilon y_2,\epsilon z_2,\epsilon):=\epsilon^{-3} W(\epsilon x_2,\epsilon (\epsilon y_2),\epsilon (\epsilon z_2),\epsilon),\quad W=F,G,H.
\end{align*}
Here $F_2,G_2,H_2$ all have real-analytic extensions to $\epsilon=0$ by \eqref{R0}. For $\epsilon=0$, we obtain $\dot x_2=x_2^2-1$, $y_2=z_2=0$, having a heteroclinic connection given by
\begin{align}\eqlab{unperturbedsol}
x_2=-\tanh(t)\to \mp 1 \mbox{ for $t\to \pm\infty $, respectively}.
\end{align} The idea of \cite{baldom2013a} is then to parameterize the unstable and stable manifolds of $E^{\pm}(\epsilon)$ in terms of $t\in \mathbb C$ up close to the poles of \eqref{unperturbedsol} at $t=\pm i\frac{\pi}{2}$. In particular, upon setting
\begin{align}\eqlab{t2defn}
t=\pm i\frac{\pi}{2}+\epsilon t_{2}
\end{align}
with $t_2=\mathcal O(1)$, a so-called inner problem is defined for $\epsilon\to 0$, and it is shown that the unstable and stable manifolds lie close to special invariant manifold solutions of this inner problem. In this way, given that the distance between the invariant manifolds is known near the poles (with $t_2$ fixed, see \eqref{t2defn}), the desired separation of the unstable and stable manifolds at $t=0$ (corresponding to $x=0$ by \eqref{unperturbedsol}) is obtained through the solution of a certain linear boundary value problem (obviously up to many technical details; see also \cite[pp. 673--674]{MR4892796} for a related summary of the approach).

The constant $C_0$ in \eqref{dist} (known as a Stokes constant) is essentially a measure of the separation of the invariant manifold solutions of the inner problem.

The approach in \cite{baldom2013a}, which is more functional-analytic in nature than dynamical-systems-oriented, has been successful in many important problems with exponentially small splitting, see e.g. \cite{MR4455359,MR4621957,MR4940205,MR4743478,gaivao2011a,MR4892796} and references therein. However, the point of departure (also within the context of maps \cite{gelfreich1999a,lazutkin2005a}) is always a parametrization of the unperturbed solution and its singularities in the complex plane, which is not only a strong assumption in a general context, but also an unusual one from the perspective of dynamical systems theory. 

\subsection{Main results}
In this paper, we provide an alternative method for the asymptotic splitting of the one-dimensional invariant manifolds. In particular, our approach does not use complex time and the explicit solution of the $\epsilon=0$ limit, but instead we work exclusively in the complexified phase space $(x,y,z)\in \mathbb C^3$, parametrizing our invariant manifolds as graphs over $x$. Our approach is therefore more in tune with dynamical systems theory.  It is our anticipation that the method will be useful in problems where explicit time parametrizations and their singularities are unknown.

At the same time, we believe that there are other advantages of our approach. For example, we directly relate $C_0\ne 0$ with the lack of analyticity of certain \textit{unperturbed invariant manifolds} of \eqref{model}$\vert_{\epsilon=0}$, see \lemmaref{m0pm}. These manifolds are graphs $(y,z)=\psi_0^{\pm}(x)$, $\psi_0^{\pm}(0)=(0,0)$, over local complex sectors $x\in S^{\pm} \subset \mathbb C$ centered along the positive (negative, respectively) real axis and with openings that are greater than $\pi$. These unperturbed manifolds correspond to the invariant manifolds of the inner problem of \cite{baldom2013a} (upon a change of coordinates defined by $x=t_2^{-1}$, see \eqref{t2defn}). Moreover, we improve the remainder $\mathcal O(1/\log \epsilon)$ in \eqref{dist}, which is not part of the problem but a result of the method of \cite{baldom2013a}. We also believe that our geometric approach for the difference has general interest; see further discussion of our approach in \secref{summary} and \secref{discussion} below.

We summarize our main results as follows (for further details, we refer to the main text below):
\begin{theorem}\thmlab{main}
    Consider \eqref{model}, with $F$, $G$ and $H$ real-analytic on a neighborhood of $(x,y,z,\epsilon)=(0,0,0,0)$,  satisfying \eqref{Wthird}, \eqref{Fexpansion}, and
     \begin{align}
\mbox{$b>0$ \, and \, $-1<\sigma <1$},\eqlab{bass}
\end{align}
    both fixed, for  $0<\epsilon\ll 1$. Suppose that the unperturbed invariant manifolds $(y,z)=\psi_0^{\pm}(x)$, $x\in S^{\pm}$, $\psi_0^{\pm}(0)=(0,0)$, of \eqref{model}$\vert_{\epsilon=0}$ are non-analytic at the origin. Then there is an $\epsilon_0>0$ small enough, such that the following holds:

Let $m^{\pm} (0,\epsilon)\in \mathbb R^2$ denote the $(y,z)$-components of the first intersection of $\mathbf{W}_{\mathrm{loc}}^u(E^+(\epsilon))$ and $\mathbf{W}_{\mathrm{loc}}^s(E^-(\epsilon))$ with $\{x=0\}$, respectively. Then $m^{\pm}(0,\epsilon)$ are well-defined for all $\epsilon \in (0,\epsilon_0)$ and
the splitting is determined by
    \begin{align}\eqlab{asymp}
    m^+(0,\epsilon)-m^-(0,\epsilon) = \epsilon^{-b} \e^{-\frac{\pi}{2}\epsilon^{-1}}
    \operatorname{Re} \begin{pmatrix}
                                       \e^{i\FF \log \epsilon^{-1}+\phi_0(\epsilon) +\epsilon \phi_1(\epsilon)} \\
                                       i\e^{i\FF \log \epsilon^{-1} +\phi_0(\epsilon)}
                                      \end{pmatrix},
    \end{align}
    where $\phi_\alpha:[0,\epsilon_0)\rightarrow \mathbb C$, $\alpha\in \{0,1\}$, are $C^\infty$-smooth functions. In particular, $m^+(0,\epsilon)\ne m^-(0,\epsilon)$ for all $\epsilon\in (0,\epsilon_0)$.
\end{theorem}
\begin{remark}
Through elementary calculations based on \eqref{asymp}, we find that the distance $d(\epsilon):=\vert m^+(0,\epsilon)-m^-(0,\epsilon)\vert$ takes the following form
 \begin{align*}
 d(\epsilon) =\epsilon^{-b} \e^{-\frac{\pi}{2}\epsilon^{-1}} \left(C_0+\mathcal O(\epsilon)\right),\quad C_0:=\e^{\operatorname{Re}(\phi_0(0))}>0.
 \end{align*}
  Here $\mathcal O(\epsilon)$ is a $C^\infty$-smooth function of $(\cos(\FF \log \epsilon^{-1}),\sin (\FF \log \epsilon^{-1}),\epsilon)$.
\end{remark}
 \begin{remark}\remlab{extension}
  The unperturbed invariant manifolds $(y,z)=\psi_0^{\pm}(x)$, $x\in S^{\pm}$, are obtained by Borel-Laplace. In particular, the manifolds are analytic if and only if the Borel transform $\mathcal B(\psi_0)(w)$ defines an entire function with at most exponential growth for $w\rightarrow \infty$, see \cite{bonckaert2008a,ksum}. However, the linear part of the equation for $\mathcal B(\psi_0)(w)$ has poles at $w=\pm i$ and this is also the general case for the actual Borel transform of $\psi_0^\pm$ (for the full nonlinear equation, see also \cite{costin2009,kristiansen2025a,martinet1983a,sauzin2015}).

  Having said that, our approach does not require us to assume that the unperturbed manifolds are nonanalytic. In fact,
    \begin{align}
    m^+(0,\epsilon)-m^-(0,\epsilon) = \epsilon^{-b} \e^{-\frac{\pi}{2}\epsilon^{-1}}
    \operatorname{Re} \begin{pmatrix}
                                       \widetilde \phi_0(\epsilon) \e^{i\FF \log \epsilon^{-1} +\epsilon \phi_1(\epsilon)} \\
                                       i \widetilde \phi_0(\epsilon)  \e^{i\FF \log \epsilon^{-1}}
                                      \end{pmatrix},
    \end{align}
    would be a more general statement for the splitting, with $\widetilde \phi_0(0)\ne 0$ if the unperturbed manifolds are nonanalytic at the origin. More generally, we have
    \begin{align*}
    \begin{cases}\frac{\partial^{\alpha}\widetilde \phi_0}{\partial \epsilon^\alpha}(0)=0,\quad \forall\,\alpha\in \{0,1\ldots,\beta-1\},\\
     \frac{\partial^{\beta}\widetilde \phi_0}{\partial \epsilon^\beta}(0)\ne 0,
    \end{cases}
    \end{align*}
 provided that $\beta\in \mathbb N$ is the smallest number for which the higher order ``corrections'' $\psi_{1,\beta}^\pm$, to the unperturbed manifold, are nonanalytic at the origin, see Lemma \ref{lem:Pnl}. These statements are simple consequences of \eqref{Deltauexp} below. We therefore speculate whether the following holds true: Either
 \begin{enumerate}
 \item the formal series expansion of the connection converges and $d(\epsilon)= 0$ for all $0<\epsilon\ll 1$,
 \end{enumerate}
 \textnormal{or}
 \begin{enumerate}[resume*]
  \item there is some $\beta\in \mathbb N_0$ and some $C_\beta\ne 0$ such that
 \begin{align*}
 d(\epsilon) =\epsilon^{-b+\beta} \e^{-\frac{\pi}{2}\epsilon^{-1}} \left(C_\beta+\mathcal O(\epsilon)\right).
 \end{align*}

 \end{enumerate}
 This conjecture relates to the existence of ``ultra-exponentially small splitting of separatrices'' as discussed in \cite{fiedler2025a}. We leave this open to future work.
 \end{remark}

\subsection{A summary of our approach}\seclab{summary}

A central technical aspect of the proof of \thmref{main}, is that we view \eqref{buhopf} as a local version of the blowup transformation
\begin{align}\eqlab{buthis}
 r\ge 0,\,(\breve x,\breve y,\breve z,\breve \epsilon )\in \mathbb S^3\mapsto \begin{cases}
                                                x = r\breve x,\\
                                                y = r^2\breve y, \\
                                                z = r^2\breve z,\\
                                                \epsilon =r\breve \epsilon.
                                               \end{cases}
\end{align}
(Here we refrain from using the standard notation (in blowup) $(\overline x,\overline y,\overline z,\overline \epsilon )\in \mathbb S^3$ for points on the sphere (see \cite{krupa_extending_2001}), since we would like to reserve the bar for conjugation of complex numbers, and therefore use $(\breve x,\breve y,\breve z,\breve \epsilon )\in \mathbb S^3$ instead).
In particular, we use the associated $\breve x=1$-chart, with chart-specific coordinates $(r_1,y_1,z_1,\epsilon_1)$ defined by
\begin{align}\eqlab{buthis1}
 \begin{cases}
                                                x = r_1,\\
                                                y = r_1^2y_1, \\
                                                z = r_1^2z_1,\\
                                                \epsilon =r_1\epsilon_1,
                                               \end{cases}
\end{align}
to extend the unstable and stable manifolds, which are naturally parameterized in compact subsets of the $(x_2,y_2,z_2)$-space as graphs $(y_2,z_2)=m_2^{\pm}(x_2,\epsilon)$ over $x_2\in \mathcal X_2^{\pm}$, see \eqref{zeroHopf}, to order $x=\mathcal O(1)\in \mathbb C$, $\vert x\vert>0$ sufficiently small. For the extension, we apply the flow in appropriate normal form coordinates, obtained by following the approach in \cite{uldall2024a}, using two separate notions of formal invariant manifolds: (i) formal series with respect to $r_1$ with $\epsilon_1$-dependent coefficients, and (ii) formal series with respect to $\epsilon_1$ with $r_1$-dependent coefficients (with $r_1\in S^\pm\subset \mathbb C$).
Our result shows that the unperturbed manifolds give the leading order approximation of the unstable and stable manifolds $(y,z)=m^{\pm}(x,\epsilon)$, for $x\in S^{\pm}$, $\vert x\vert>0$ small enough, respectively, as $\epsilon\to 0$. Therefore, whenever the unperturbed manifolds are distinct on the common domain $S^+\cap S^-$, we obtain an $\mathcal O(1)$-separation of the unstable and stable manifolds for  $x\in S^+\cap S^-$, $\vert x\vert \ge \xi>0$ small enough, for $0<\epsilon\ll 1$. To determine the separation of the invariant manifolds at $\{x=0\}$, we then derive a linear equation for the difference $(\Delta y,\Delta z) = (m^+-m^-)(x,\epsilon)$ (as in \cite{baldom2013a}) for $x$ on the imaginary axis $x\in i \mathbb R$, $\vert x\vert\ge 0$ small enough. As a novel aspect, we identify this set of equations as a slow-fast system with a normally hyperbolic critical manifold (of saddle-type) for $\epsilon\to 0$. Therefore, in contrast to the functional-analytic approach of \cite{baldom2013a}, we use Geometric Singular Perturbation Theory (GSPT) \cite{fen3,jones_1995} to study this linear system and derive a Fenichel-type normal form \cite[Section 5]{jones_1995}, bringing the linear system to diagonal form. In this way, the separation can easily be integrated from $x=\pm i \delta$, $\delta>0$ fixed small enough, to $x=0$. Here we also use that the equations are real-analytic: $(x,y,z)(t)$ is a solution of \eqref{model}, $\epsilon>0$, if only if $(\overline x,\overline y,\overline z)(t)$ is a solution, where the bar denotes the complex conjugate. We illustrate our approach in \figref{blowup}.

We believe that our approach to study the difference $(\Delta y,\Delta z)$ has general interest. Besides the  zero-Hopf bifurcations of higher co-dimensions discussed below, we are presently pursuing the same approach to generalize the results in \cite{hayes2016a} to general delayed-Hopf phenomena (see \cite{neishtadt1987a,neishtadt1988a}).  

 \begin{figure}[h!]
\begin{center}
{\includegraphics[width=.85\textwidth]{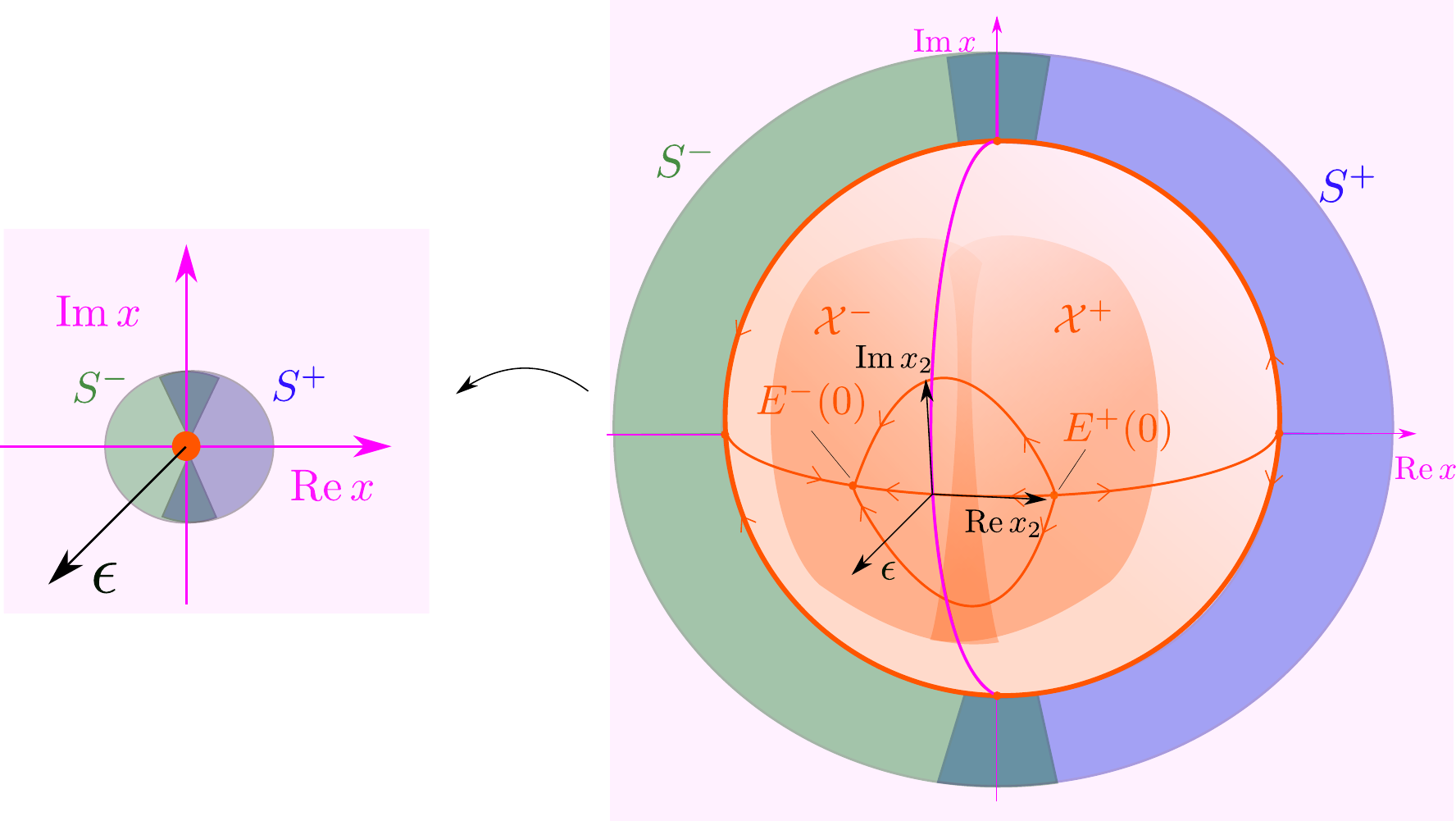}}
\end{center}
\caption{Illustration of our approach. We blow up $(x,y,z,\epsilon)=(0,0,0,0)$ to a complex $3$-sphere $\mathbb S^3$, see \eqref{buthis}, which we here indicate in the $(\operatorname{Re}(x),\operatorname{Im}(x),\epsilon)$-projection. On top of the sphere ($\breve \epsilon>0$), we use the (scaled) coordinates $(x_2,y_2,z_2)$, see \eqref{buhopf}, to describe the existence of the saddle-foci $E^{\pm}(\epsilon)$, $0<\epsilon\ll 1$, and their one-dimensional invariant manifolds as graphs  $(y,z)=m^{\pm}(x,\epsilon)$ over $x_2=\epsilon^{-1} x\in \mathcal X^{\pm}\subset \mathbb C$, for all $0<\epsilon\ll 1$. The orange curves are orbits of $\dot x_2=x_2^2-1$, $x_2\in \mathbb C$, and $\mathcal X^\mp$ is any fixed compact set in the basin of attraction for $x_2=-1$ and $x_2=1$ (for the backward flow), respectively. A central objective of our approach is to extend the unstable and stable manifolds of $E^{\pm}(\epsilon)$, respectively, by working in the coordinates $(r_1,y_1,z_1,\epsilon_1)$, see \eqref{r1y1z1eps10}, so that they can be compared with the unperturbed invariant manifolds that are defined as graphs over $x\in S^{\pm}$ (blue and green, respectively) within $\epsilon=0$. Our results show that if the unperturbed manifolds are non-analytic, then there is an $\mathcal O(1)$-splitting of the unstable and stable manifolds for $x\in S^+\cap S^-$, $\vert x\vert>0$ small enough, as $\epsilon\to 0$. Finally, we integrate the difference $(\Delta y,\Delta z)=(m^+-m^-)(x,\epsilon)$ down along the imaginary axis $x\in i\mathbb R$. In this step, we view the equation for the difference as a slow-fast system with a normally hyperbolic critical manifold, which we then treat by GSPT (through a Fenichel-type normal form).}
\figlab{blowup}
\end{figure}

\subsection{Discussion}\seclab{discussion}
An interesting aspect of our approach is that we do not use desingularization in association with our blowup. Essentially blowup is only used to compactify the $(x_2,y_2,z_2)$-space. Hence we expect that our approach can also be applied to splitting problems associated with discrete time systems (where desingularization has been an obstruction in general, see \cite{MR4814277}).

In \cite{new}, which can be seen as a follow-up paper, the present author studies an $\epsilon$-family of systems with degenerate zero-Hopf bifurcations for $\epsilon=0$ of arbitrary co-dimension $\kappa\in \mathbb N$, $\kappa\ge 3$, using the same approach. Here we find $2(\kappa-1)$ unperturbed invariant manifolds and the separation of these manifolds on common domains lead to exponentially small phenomena of $\kappa-1$ (formal) heteroclinic connections. In \cite{new}, we therefore generalize the last step of the method for dealing with the equation for the differences $(\Delta y,\Delta z)(x,\epsilon)$. For this, we use the notion of elliptic and hyperbolic paths in $x\in \mathbb C$ in a systematic way (these curves are also known as Stokes and anti-Stokes paths, respectively, see e.g. \cite{neishtadt1988a}). Such paths $x(s)\in \mathbb C$, $s\in \mathbb R$, in the complex plane are characterized by the linear part of the $(y,z)(s)$-subsystem being hyperbolic (expanding/contracting) and elliptic (oscillatory), respectively. In the present case of \eqref{model}, these paths correspond to orbits of $\dot x_2=x_2^2-1$ and $\dot x_2=i(x_2^2-1)$, respectively. The results of \cite{new} show that the $j$th splitting, $j\in \{1,\ldots,\kappa-1\}$, takes the form $$\epsilon^{e(\kappa)}\e^{-\epsilon^{1-\kappa}T_j } \left(C_j+\mathcal O(\epsilon)\right), \quad T_j>0,\,e(\kappa)\in \mathbb R,$$ under a nondegeneracy condition. Here $\{T_j\}_{j=1}^{\kappa-1}$ are finite blowup times for $(\kappa-1)$-many special unbounded  solutions $\{\mathcal H_j\}_{j=1}^{\kappa-1}\subset \mathbb C$, respectively, of the reduced problem in imaginary time: $\dot x_2 = iQ_2(x_2)$ with $Q_2$ a polynomial of degree $\kappa$. All roots $\{q_j\}_{j=1}^{\kappa}$ of $Q_2$ are assumed to be real and simple. Then $T_j>0$ is expressed as
\begin{align*}
T_j = \left|\sum_{l=1}^j \frac{\pi}{Q'(q_l)}\right|.
\end{align*}

We refer to \cite{new} for further details. For $\kappa=2$ and $Q_2(x_2)=x_2^2-1$, one has $\mathcal H_1=i\mathbb R$ and $T_1=\frac{\pi}{2}$ in agreement with \eqref{dist}. Now, the blowup in imaginary time is intrinsically related to singularities of time parametrizations of connecting orbits, see \cite{fiedler2025a}, but the time-parametrizations are not used explicitly in \cite{new}; the construction is purely based upon phase space methods of dynamical systems theory.

 For the generalization of our approach in \cite{new} to higher co-dimension, we would like mention \cite{hayes2016a,neishtadt1987a,neishtadt1988a} as inspirations. In particular, the reference \cite{hayes2016a} used GSPT \cite{fen3,jones_1995} and blowup \cite{dumortier_1996,krupa_extending_2001} to describe exponentially small phenomena in simple model systems of the form $\epsilon z'(t) = (t+i) z + \epsilon h(t,z,\epsilon)$ for slow passage through a Hopf bifurcation using hyperbolic paths. On the other hand, \cite{neishtadt1987a,neishtadt1988a} used elliptic paths to estimate the delay of general systems with slow passage through a Hopf bifurcation.

We emphasize that there is a different (but related) notion  of Stokes and anti-Stokes lines of (exponential) asymptotic expansions of analytic functions, see \cite{berry1988a,chapman1998a}. Recently, methods from exponential asymptotics have been used to describe exponentially small phenomena in nonlinear equations, see e.g. \cite{chapman1998a,CHAPMAN2009319,6a948bd44bf4469e875c1f1482ef7619} and references therein. Although this approach appears to be more formal, it has been extremely successful. In fact, the results of \cite{CHAPMAN2009319}, which show excellent agreement with numerical computations, presently seem to escape rigorous methods.  The reference \cite{chapman1998a} has a summary of the approach: First, solutions are expanded in formal power series with respect $\epsilon$, having $t$-dependent coefficients $\phi_\alpha(t)$, $\alpha\in \mathbb N$. Next, one applies a WKB-type ansatz for the behaviour of $\phi_\alpha$ as $\alpha\to \infty$. This in turn induces an optimal truncation of the asymptotic series from which the dominant term of the exponentially small terms can be determined. A similar approach was used in the influential paper \cite{kruskal1991a} by Kruskal and Segur on nonexistence of connecting orbits in a model for crystal growth.

Finally, we would like to mention \cite{MR4855745} as an inspiration for this work. This paper studies weak analytic invariant manifolds of hyperbolic nodes close to saddle-nodes, which also relates to exponentially small phenomena. In summary, \cite{MR4855745} shows (following \cite{MR4445442}) that if the center manifold at the saddle-node is non-analytic, then the weak analytic manifolds can be controlled, see also \cite{rousseau2005a}. The unperturbed invariant manifolds in the present paper are also center-like manifolds, being graphs over the zero-eigenspace. I therefore find that \cite{MR4855745} bears some important similarities with the present paper.

\subsection{Overview}
The paper is organized as follows: In \secref{unperturbed}, we consider the $\epsilon=0$ limit and describe the unperturbed invariant manifolds, see \lemmaref{m0pm}. Subsequently, in \secref{model2} we consider \eqref{zeroHopf} with $0<\epsilon\ll 1$ and describe the equilibria $E^{\pm}(\epsilon)$ and their one-dimensional invariant manifolds as graphs over large but fixed compact domains $x_2\in \mathcal X_2^{\pm}\subset \mathbb C$ (corresponding to $x=\mathcal O(\epsilon)$ by \eqref{buhopf}) for $0<\epsilon\ll 1$, see \propref{unstable2}. In \secref{blowup}, we describe our blowup approach. In particular, in this section we extend the unstable and stable manifolds to $x=\mathcal O(1)$, by working in the $\breve x=1$-chart associated with \eqref{buthis}, see \eqref{buthis1}, and compare these manifolds with the unperturbed ones, see \propref{prop3}. Finally, in \secref{difference} we
complete the proof of \thmref{main} by setting up an equation for the difference $(\Delta y,\Delta z)(x,\epsilon)$ for $x\in i\mathbb R$. We identify this equation as a slow-fast system and derive a Fenichel-type normal form. From this normal form, we complete the proof of \thmref{main}, see \secref{completing}.

\section{The $\epsilon=0$-system: The unperturbed invariant manifolds}\seclab{unperturbed}
In this section, we consider the unperturbed system:
\begin{equation}\eqlab{model00}
\begin{aligned}
 x' &=x^2+a (y^2 +z^2) + F(x,y,z,0),\\
 y'&= -bx y + z+G(x,y,z,0),\\
 z'&= -bx z - y+H(x,y,z,0).
\end{aligned}
\end{equation}
obtained from \eqref{model0} by setting $\epsilon=0$ (or equivalently from \eqref{model0} by setting $\mu=\nu=0$).
Let $S^{\pm}(\delta,\chi)\subset \mathbb C$ denote the local open sectors, centered along the positive ($+$) (negative $(-$), respectively) axis, of radius $\delta>0$ and opening $\pi+\chi\in (\pi,2\pi)$:
\begin{align}
 S^{\pm}(\delta,\chi)=\left\{x\in \mathbb C\,:\,0< \vert x\vert< \delta\, \mbox{ and }\, \vert \operatorname{Arg}(\pm x)\vert<  \frac{\pi+\chi}{2}\right\},\nonumber
\end{align}
respectively,
see \figref{Spm}(a). Notice that $0<\chi<\pi$.
\begin{lemma}\lemmalab{m0pm}
Consider \eqref{model00}. Then for $\delta>0$ sufficiently small, there exist two unique local invariant manifolds of \eqref{model00} as graphs over $S^{\pm}=S^{\pm}(\delta,\chi)$:
\begin{align}
(y,z)=\psi_0^{\pm} (x), \quad x\in S^{\pm}.\eqlab{psi0omg}
\end{align}
Here $\psi_0^{\pm}:S^{\pm}\rightarrow \mathbb C^2$ are each real-analytic and $\psi_0^{\pm} (x)=\mathcal O(x^3)$ uniformly on the closure of $S^{\pm}$. In particular,
they are the $1$-sums of a Gevrey-$1$ series:
\begin{align}
 \psi_0^{\pm}(x) \sim_1 \sum_{\alpha=3}^\infty \psi_{0,\alpha} x^\alpha, \quad \vert \psi_{0,\alpha}\vert \le c_1 c_2^\alpha \alpha! \quad \forall\,\alpha\ge 3,\eqlab{gevrey1series}
\end{align}
for some $c_1,c_2>0$, in the directions defined by $S^{\pm}$, respectively. In particular, the series in \eqref{gevrey1series} is independent of $\pm$.
\end{lemma}
\begin{proof}
%
We define $(r_1,y_1,z_1)$ by
\begin{align}\eqlab{this1}
 \begin{cases}
  x = r_1,\\
  y = r_1^2 y_1,\\
  z =r_1^2 z_1.
 \end{cases}
\end{align}
This gives
\begin{equation}\eqlab{r1y1z1}
\begin{aligned}
r_1' &= r_1^2 (1+\mathcal O(r_1)),\\
\zeta_1'&=\Omega \zeta_1+\mathcal O(r_1),
\end{aligned}
 \end{equation}
where
\begin{align}\eqlab{Omegadefn}
 \Omega:= \begin{pmatrix}
          0 &1\\
          -1 &0
         \end{pmatrix},
         \end{align}
upon setting $$\zeta_1= (y_1,z_1)^{\mathrm{T}},$$
or as
a first order real-analytic system
\begin{align*}
r_1^2 \frac{d\zeta_1}{dr_1} &= \Omega \zeta_1 + \mathcal O(r_1).
\end{align*}
This system is a generalized saddle-node with Poincar\'e rank $1$ studied by many authors, see e.g. \cite{bonckaert2008a}. In particular, the linearization of the right hand side for $r_1=0$ at $\zeta_1=(0,0)$ is given by $\Omega$, having nonzero eigenvalues $\pm i$.
It then follows from \cite[Theorem 3]{bonckaert2008a} (as a corollary of the invariance of $y=0$ for the normal form \cite[Eq. (8)]{bonckaert2008a}, see also \cite[Lemma 4.8]{uldall2024a} or
\cite[Corollary 4.2]{ksum} where the statement is made more explicit) that there exists two solutions
\begin{align}
 \zeta_1 = \psi^{\pm}_{1,0}(r_1),\quad r_1\in S^{\pm},\eqlab{zeta1psi}
\end{align}
with $\psi_{1,0}^{\pm}(0)=0$ and $S^{\pm}$ as in the statement. In particular, $\psi^{\pm}_{1,0}:S^{\pm}\rightarrow \mathbb C^2$ are each real-analytic, the $1$-sum of a Gevrey-$1$ series, with $\psi^{\pm}_{1,0}(r_1)=\mathcal O(r_1)$ uniformly on the closure of $S^{\pm}$, respectively.  We complete the theorem by setting  $\psi_0^{\pm}(r_1): = r_1^2 \psi^{\pm}_{1,0}(r_1)$, cf. \eqref{this1}.
\end{proof}

It is clear that $(y,z)=\psi_0^{\pm}(x)$, $x\in S^{\pm} \cap \mathbb R$, define unstable/stable sets, respectively, of the zero-Hopf point $(x,y,z)=(0,0,0)$ for $\epsilon=0$, see \figref{Spm}(b).

\textit{We refer to the invariant manifolds $(y,z)=\psi^{\pm}_0 (x)$, $x\in S^{\pm}$, of \eqref{model}, as the unperturbed manifolds}.

  \begin{figure}[h!]
\begin{center}
\subfigure[]{\includegraphics[width=.45\textwidth]{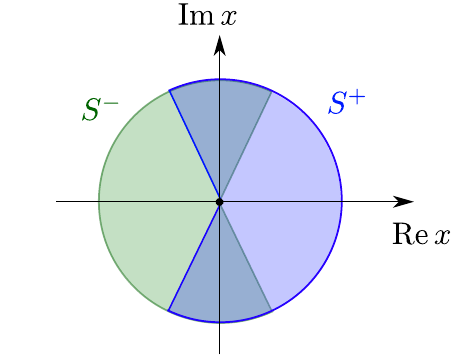}}
\subfigure[]{\includegraphics[width=.45\textwidth]{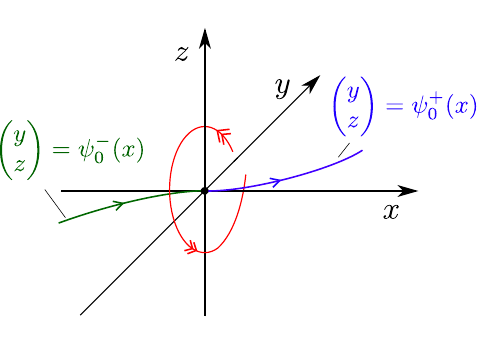}}
\end{center}
\caption{In (a): The sectors $S^{\pm}$ in the complex plane. In (b): The invariant graphs \eqref{psi0omg} of \eqref{model} for $\epsilon=0$ in the real $(x,y,z)$-space.}
\figlab{Spm}
\end{figure}

\section{The perturbed system: The unstable and stable manifolds}\seclab{model2}
In this section, we consider \eqref{zeroHopf} for all $0<\epsilon\ll 1$, repeated here in the fast form
\begin{equation}\eqlab{model2}
\begin{aligned}
 x_2' &=\epsilon \left(x_2^2-1+a \epsilon^2 (y_2^2 +z_2^2)+\epsilon F_2(x_2,\epsilon y_2,\epsilon z_2,\epsilon)\right),\\
 y_2'&= \epsilon b(-x_2+\sigma) y_2 + z_2+\epsilon G_2(x_2,\epsilon y_2,\epsilon z_2,\epsilon),\\
 z_2'&= \epsilon b( -x_2+\sigma) z_2 - y_2+\epsilon H_2(x_2, \epsilon y_2,\epsilon z_2,\epsilon),
\end{aligned}
\end{equation}
with $(\cdot)'=\frac{d}{d\tau}$ and where $F_2,G_2,H_2$ are defined by
\begin{align*}
W_2 (x_2,\epsilon y_2,\epsilon z_2,\epsilon):= \epsilon^{-3} W(\epsilon x_2,\epsilon (\epsilon y_2),\epsilon (\epsilon z_2), \epsilon),\quad W=F,G,H.
\end{align*}
Notice that $F_2,G_2$ and $H_2$ have analytic extensions to $\epsilon=0$ by \eqref{R0}. In particular, as $F$, $G$ and $H$ are assumed to be real-analytic on a neighborhood of $(x,y,z,\mu)=(0,0,0,0)$, it follows that $F_2$, $G_2$ and $H_2$ are analytic on a polydisc:
\begin{align}\eqlab{polydisc}
(x_2,\epsilon y_2,\epsilon z_2,\epsilon)\in (\epsilon^{-1} B_\xi) \times (\epsilon^{-1} B_\xi)\times (\epsilon^{-1} B_\xi) \times B_\xi,
\end{align}
 for $\xi>0$ small enough. Here $B_\xi\subset \mathbb C$ denotes the open disc of radius $\xi>0$ centered at the origin. We recall that \eqref{model2} is obtained from \eqref{model} by the change of coordinates
\begin{align}\eqlab{x2y2z2}
 (x_2,y_2,z_2)\mapsto \begin{cases} x =\epsilon x_2,\\
  y = \epsilon^2 y_2,\\
  z =\epsilon^2 z_2.
 \end{cases}
\end{align}
The system \eqref{zeroHopf} is a slow-fast system for $0<\epsilon\ll 1$, with $x_2$ being slow and $(y_2,z_2)$ being fast. For $\epsilon=0$, we obtain the layer problem
\begin{equation}\eqlab{layer2}
\begin{aligned}
 x_2'&=0,\\
 y_2'&= z_2,\\
 z_2'&= -y_2,
\end{aligned}
\end{equation}
with $(x_2,0,0)$ defining a critical manifold, which is normally elliptic. In terms of the slow time $t=\epsilon \tau$, we have
\begin{equation}\eqlab{zeroHopfslow}
\begin{aligned}
 \dot x_2 &=x_2^2-1+a \epsilon^2 (y_2^2 +z_2^2)+\epsilon F_2(x_2,\epsilon y_2,\epsilon z_2,\epsilon),\\
 \epsilon \dot y_2&= \epsilon b(-x_2+\sigma) y_2 + z_2+\epsilon G_2(x_2,\epsilon y_2,\epsilon z_2,\epsilon),\\
 \epsilon \dot z_2&= \epsilon b( -x_2+\sigma) z_2 - y_2+\epsilon H_2(x_2, \epsilon y_2,\epsilon z_2,\epsilon),
\end{aligned}
\end{equation}
as in \eqref{zeroHopf},
which for $\epsilon=0$ gives the reduced problem:
\begin{align}\eqlab{reduced}
 \dot x_2  =Q_2(x_2):= x_2^2-1,
\end{align}
for $y_2=z_2=0$. We have hyperbolic singularities at $x_2=\pm 1$, being unstable and stable, respectively.
We illustrate the dynamics of \eqref{reduced}, $x_2\in \mathbb C$, on the Poincar\'e sphere  \cite{dumortier2006a} in \figref{reduced}. The compactification is based upon
\begin{align*}
 \begin{cases}\operatorname{Re}(x_2) = {\breve u}{\breve w}^{-1},\\
 \operatorname{Im}(x_2) = {\breve v}{\breve w}^{-1},
 \end{cases}
\end{align*}
with $(\breve u,\breve v,\breve w)\in \mathbb H^2$. Here $\mathbb H^2\subset \mathbb R^3$ denotes the hemisphere:
\begin{align*}
 \mathbb H^2:=\{(\breve u,\breve v,\breve w)\in \mathbb R^3\,:\,\breve u^2+\breve v^2+\breve w^2=1,\,\breve w\ge 0\}.
\end{align*}
The equator circle $\mathbb S^1 = \mathbb H^2\cap \{\breve w=0\}$ represents $x_2=\infty$.
There is a one-to-one correspondence between $x_2\in \mathbb C$ and $(\breve u,\breve v,\breve w)\in \operatorname{int}\mathbb H^2:=\mathbb H^2\cap \{\breve w>0\}$. As usual in Poincar\'e compactification, we use a desingularization corresponding to division of the right hand side by $1+\vert x_2\vert^2$ to ensure that the flow is well-defined on the equator circle.

 \begin{figure}[h!]
\begin{center}
{\includegraphics[width=.475\textwidth]{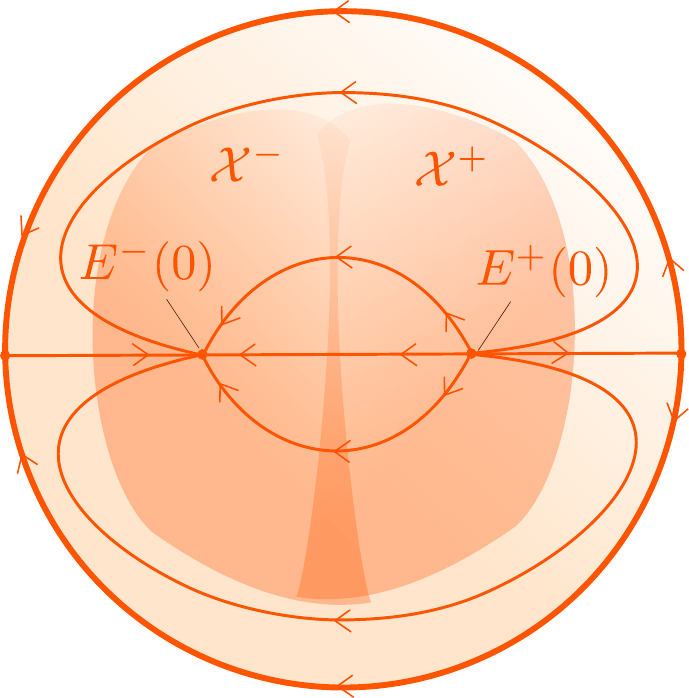}}
\end{center}
\caption{Phase portrait of $x_2'=x_2^2-1$ on the Poincar\'e sphere (here represented as a disc). The boundary circle $\mathbb S^1$ represents $x_2=\infty$. The compact domains $\mathcal X^{\mp}$, used in \propref{unstable2}, are subsets of the basin of attraction of $x_2=\mp 1$ for the forward and backward flow, respectively. }
\figlab{reduced}
\end{figure}

\begin{lemma}\lemmalab{Epm}
 Consider \eqref{model2} with $0<\epsilon\ll 1$ and suppose \eqref{bass}. Then there exist two hyperbolic  saddle-foci $E^{\pm} (\epsilon)$ given by
 \begin{align*}
  (x_2,y_2,z_2)=(X^{\pm}({\epsilon}),\epsilon  Y^{\pm}(\epsilon ),\epsilon Z^{\pm} ({\epsilon})),\quad X^\pm(0)=\pm 1,
 \end{align*}
 for all $0<\epsilon<\epsilon_0$, $0<\epsilon_0\ll 1$, respectively. Here $X^{\pm},Y^{\pm},Z^{\pm}$ are all real-analytic with respect to ${\epsilon}\in [0,\epsilon_0)$. Moreover, $E^{+}(\epsilon)$ has a one-dimensional unstable manifold $\mathbf{W}^u(E^+(\epsilon))$ and a two-dimensional stable manifold whereas $E^{-}(\epsilon)$ has a one-dimensional stable manifold $\mathbf{W}^s(E^-(\epsilon))$ and a two-dimensional unstable manifold.
\end{lemma}
\begin{proof}
Singularities of \eqref{model2} are given by 
\begin{align*}
  0 &=x_2^2-1+a \epsilon^2 (y_2^2 +z_2^2) + \epsilon F_2(x_2,\epsilon y_2,\epsilon z_2,\epsilon),\\
 0&= \epsilon b(-x_2+ \sigma) y_2 + z_2+\epsilon G_2(x_2,\epsilon y_2,\epsilon z_2,\epsilon),\\
 0&= \epsilon b( -x_2+ \sigma) z_2 - y_2+\epsilon H_2(x_2,\epsilon y_2,\epsilon z_2,\epsilon).
\end{align*}
%
The existence and expansion of $E^{\pm}$ then follow from the implicit function theorem, using the regular solutions $(x_2,y_2,z_2)=(\pm 1,0,0)$ for $\epsilon=0$. Subsequently, the statements regarding stability follow from simple calculations using \eqref{bass}.
\end{proof}
We illustrate the results of \lemmaref{Epm} in \figref{x2y2z2}.
  \begin{figure}[h!]
\begin{center}
\subfigure[]{\includegraphics[width=.45\textwidth]{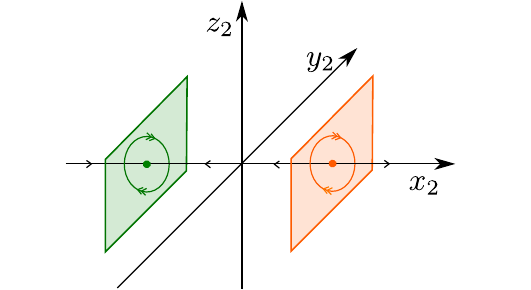}}
\subfigure[]{\includegraphics[width=.45\textwidth]{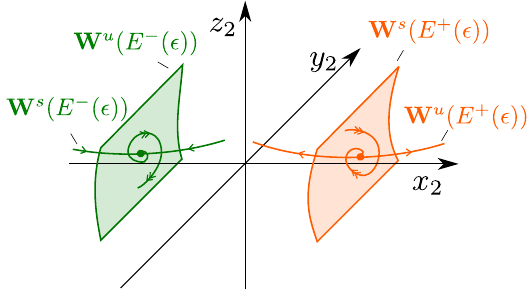}}
\end{center}
\caption{Illustration of the results of \lemmaref{Epm}. In (a): $\epsilon=0$. In (b): $0<\epsilon\ll 1$. }
\figlab{x2y2z2}
\end{figure}
%
In the following, we describe the local invariant manifolds $\mathbf{W}_{loc}^s(E^-(\epsilon))$ and $\mathbf{W}_{loc}^u(E^+(\epsilon))$ as graphs $(y_2,z_2)=m_2^\pm(x_2)$ over $x_2$. These manifolds are solutions of the invariance equation:
\begin{align}\eqlab{invman0}
 \left(\Omega +\epsilon b(-x_2+\sigma) \operatorname{Id}\right)\begin{pmatrix}
                                       y_2\\
                                       z_2
                                      \end{pmatrix}+\epsilon \begin{pmatrix}
                                      G_2\\
                                      H_2\end{pmatrix}=\epsilon \left(x_2^2-1+a\epsilon^2 (y_2^2+z_2^2)+\epsilon F_2 \right) \frac{d}{dx_2}\begin{pmatrix}
                                       y_2\\
                                       z_2
                                      \end{pmatrix},
\end{align}
where $W_2=W_2(x_2,y_2,z_2,\epsilon)$, $W=F,G,H$, and
\begin{align*}
\operatorname{Id}=\operatorname{diag}(1,1)\in \mathbb C^{2\times 2}.
\end{align*}
Now, we recall the following:
\begin{lemma}\lemmalab{formal}
There exists a unique formal series solution of \eqref{invman0} of the form
\begin{align}\label{eq:m2n}
  (y_2,z_2) = \sum_{\alpha=1}^\infty  m_{2,\alpha}( x_2)\epsilon^\alpha,
\end{align}
with $ m_{2,\alpha}:\mathbb C\rightarrow \mathbb C^2$ real-analytic for all $\alpha\in \mathbb N$.
\end{lemma}
\begin{proof}
 By \eqref{polydisc}, we have that for any $\varrho>0$, there is an $\epsilon_0>0$ such that  $F_2$, $G_2$ and $H_2$ are real-analytic on $(x_2,y_2,z_2,\epsilon) \in B_\varrho^3 \times (-\epsilon_0,\epsilon_0)$. Here $B_\varrho^3\subset \mathbb C^3$ is the ball of radius $\varrho>0$ centered at the origin. But then \cite[Proposition 2.1]{de2020a} gives the existence and uniqueness of the formal series \eqref{m2n} with $ m_{2,\alpha}:B_\varrho \rightarrow \mathbb C^2$ real-analytic for all $\alpha\in \mathbb N$ and any $\varrho>0$. The result therefore follows.
\end{proof}

We now fix any compact neighborhood $\mathcal X_{2}^-\subset \mathbb C$ of $x_2=-1$ within its basin of attraction for the forward flow of \eqref{reduced}. We similarly fix any compact neighborhood $\mathcal X_2^+\subset \mathbb C$ of $x_2=1$ within its basin of attraction for the backward flow of \eqref{reduced}. See \figref{reduced}. Then we have the following:

\begin{proposition}\proplab{unstable2}
 There is an $\epsilon_0>0$ small enough, such that the local unstable and stable manifolds, $\mathbf{W}_{\mathrm{loc}}^u(E^+(\epsilon))$ and $\mathbf{W}_{\mathrm{loc}}^s(E^-(\epsilon))$, are graphs over $x_2\in \mathcal X_{2}^{\pm}$, respectively:
 \begin{align*}
(y_2,z_2) &= m^{\pm}_2(x_2,\epsilon),\quad x_2\in \mathcal X_{2}^{\pm},
 \end{align*}
 with $\epsilon (Y_2^{\pm} ,Z_2^{\pm}) = m_2^{\pm}(X_2^{\pm},\epsilon)$, for any $\epsilon \in(0,\epsilon_0)$.
Each $m_2^{\pm}(\cdot,\epsilon):\mathcal X_2^{\pm}\rightarrow \mathbb C^2$ is real-analytic and $C^\infty$ with respect to $\epsilon\in [0,\epsilon_0)$. In particular, the functions $m_2^{\pm}$ are both asymptotic to the formal series \eqref{m2n} as $\epsilon\rightarrow 0$:
 \begin{align}\eqlab{expansion}
     m^{\pm}_2(x_2,\epsilon) \sim \sum_{\alpha=1}^\infty m_{2,\alpha}(x_2) \epsilon^\alpha,
 \end{align}
 uniformly with respect to $x_2\in \mathcal X_2^{\pm}$.
\end{proposition}
\begin{proof}
The graph form of the invariant manifolds follows from standard hyperbolic theory, but due to the singular nature of $\epsilon\to 0$, this only gives local and nonuniform statements. Nevertheless, to prove the statement, we will follow the standard proof of the stable manifold theorem through the variation of constant formulation, after some initial preparation of the equations. We focus on $m_2^-$ and the stable manifold of $E^-$. The case of $m_2^+$ is similar and details of this case are therefore left out.

To prepare the equations, we first define
%
\begin{align*}
 \begin{cases}
  x_2 =:  X_2^-(\epsilon)+\widetilde x_2,\\
  y_2 =:\epsilon Y_2^-(\epsilon)+\widetilde y_2,\\
  z_2 = :\epsilon Z_2^-(\epsilon)+\widetilde z_2,
 \end{cases}
\end{align*}
recall \lemmaref{Epm},
and
\begin{align*}
\zeta_2:= \widetilde x^{-1} \begin{pmatrix}
                          \widetilde y_2\\
                          \widetilde z_2
                         \end{pmatrix}.
                         \end{align*}
                         A simple calculation shows that $(\widetilde x_2,\zeta_2)\mapsto (x_2,y_2,z_2)$ brings \eqref{model2} into the following form by pull-back
\begin{equation}\label{eq:x2zeta2}
\begin{aligned}
 \dot{\widetilde x}_2 &= \widetilde Q_2(\widetilde x_2,\epsilon) \widetilde x_2,\\
 \epsilon \dot \zeta_2 &=\left(\Omega + \epsilon W_{21}(\widetilde x_2,\epsilon)\right) \zeta_2 + \epsilon W_{20}(\widetilde x_2,\epsilon)+\epsilon^2W_{22}(\widetilde x_2,\zeta_2,\epsilon)(\zeta_2,\zeta_2),
\end{aligned}
\end{equation}
recall \eqref{Omegadefn},
after dividing the right hand side by a real-analytic quantity $1+\mathcal O(\epsilon^{2})$, uniformly with respect to $-1+\widetilde x_2\in \mathcal X_2^-$, $\vert \zeta_2\vert \le c$, with $c>0$ fixed, and all $0\le\epsilon\ll 1$; notice that such division is justified by the fact that we are interested in invariant manifold solutions $\zeta_2 = \zeta_2(\widetilde x_2,\epsilon)$,  satisfying the corresponding invariance equation:
\begin{equation}\eqlab{invman}
\begin{aligned}
 \left(\Omega + \epsilon W_{21}(\widetilde x_2,\epsilon)\right) \zeta_2 + \epsilon  W_{20}(\widetilde x_2,\epsilon)+\epsilon^2W_{22}(\widetilde x_2,\zeta_2,\epsilon)(\zeta_2,\zeta_2) =\epsilon \widetilde Q_2(\widetilde x_2,\epsilon) \widetilde x_2 \frac{\partial \zeta_2}{\partial \widetilde x_2}.
\end{aligned}
\end{equation}
In the derivation of \eqref{x2zeta2}, we have Taylor-expanded the right hand side of the $\zeta_2$-equation around $\zeta_2=0$ and used that $Q_2(x_2):=x_2^2-1\ne 0$ for all $x_2\in \mathcal X_2^-\setminus \{-1\}$ and that $Q_2'(-1)=-2\ne 0$.

All quantities in \eqref{x2zeta2} are real-analytic. In particular, we have
\begin{align*}
 \widetilde Q_2(\widetilde x_2,0) = \widetilde x_2^{-1} Q_2(-1+\widetilde x_2)=-2+\widetilde x_2,
\end{align*}
and
\begin{align*}
 W_{21}(\widetilde x_2,0) =
\left( b(1-\widetilde x_2+\sigma) + 2-\widetilde x_2\right) \operatorname{Id},
\end{align*}
specifically
\begin{align}\label{eq:cond}
  W_{21}(0,0) = (b(1+\sigma)+2)\operatorname{Id},
\end{align}
with $b(1+\sigma)+2>0$. Finally, $W_{22}(\widetilde x_2,\zeta_2,\epsilon)(\cdot,\cdot)$ is bilinear.



As in \lemmaref{formal}, we have a unique formal series solution of \eqref{invman} of the form
\begin{align}\label{eq:zeta2n}
  \zeta_2 = \sum_{\alpha=1}^\infty \zeta_{2,\alpha}(\widetilde x_2)\epsilon^\alpha,
\end{align}
with $\zeta_{2,\alpha}:\mathbb C\rightarrow \mathbb C$ real-analytic for all $\alpha\in \mathbb N$. By uniqueness of the formal series, \eqref{zeta2n} corresponds to \eqref{m2n} upon the change of coordinates.
We therefore fix $N\in \mathbb N$ and introduce ${\widetilde \zeta}_2$ by
\begin{align}
 \zeta_2 = \sum_{\alpha=1}^{2N} \zeta_{2,\alpha}(\widetilde x_2)\epsilon^\alpha+\epsilon^N \widetilde{ \zeta}_2.\eqlab{tildezeta2defn}
\end{align}
This brings the equations into the final form
\begin{equation}\label{eq:x2zeta2f}
\begin{aligned}
 \dot{\widetilde x}_2 &= \widetilde Q_2(\widetilde x_2,\epsilon) \widetilde x_2,\\
 \epsilon \dot{\widetilde{\zeta}}_2 &=\left(\Omega + \epsilon \widetilde W_{21}^N (\widetilde x_2,\epsilon)\right) \widetilde \zeta_2 + \epsilon^{N} \widetilde W_{20}^N(\widetilde x_2,\widetilde{ \zeta}_2,\epsilon),
\end{aligned}
\end{equation}
with $\widetilde W_{21}^N(0,0)=(b(1+\sigma)+2)\operatorname{Id}$. Here $\widetilde W_{21}^N$ and $\widetilde W_{20}^N$ are real-analytic. This follows from a simple calculation.

We will prove the statement by working on \eqref{x2zeta2f}. In particular, we define $\widetilde{\mathcal X}_2^-\subset \mathbb C$ by $\widetilde x_2\in \widetilde{\mathcal X}_2^-$ if and only if $-1+\widetilde x_2\in \mathcal X_2^-$ and look for invariant manifold solutions $\widetilde \zeta_2=\widetilde \zeta_2(\widetilde x_2,\epsilon)$ as graphs over $\widetilde x_2\in \mathcal X_2^-$.

We therefore take $\widetilde x_{20}\in \widetilde{\mathcal X}_2^-$ and consider the solution $\widetilde x_2(t,\widetilde x_{20},\epsilon)$ of the initial value problem
$$\dot{\widetilde x}_2(t)=\widetilde Q_2(\widetilde x_2(t),\epsilon)\widetilde x_2(t),\quad \widetilde x_2(0)=\widetilde x_{20}.$$ By assumption, $\mathcal X_2^-$ is compact and a subset of the basin of attraction of the hyperbolic point $x_2=-1$ for the forward flow of \eqref{reduced} for $\epsilon=0$. We therefore conclude that $\widetilde x_2(t,\widetilde x_{20},\epsilon)\rightarrow 0$ for $t\rightarrow \infty$, uniformly for $\widetilde x_{20}\in \mathcal X^-$, $0\le \epsilon\ll 1$.

We then solve for $\widetilde \zeta_{2}(t,\widetilde x_{20},\epsilon)$, $t\ge 0$, such that $(\widetilde x_{20},\widetilde \zeta_{2}(0,\widetilde x_{20},\epsilon))$ is a point on the stable manifold through the usual fixed-point formulation of \eqref{x2zeta2f}:
\begin{align}\eqlab{fixpoint}
\widetilde \zeta_2(t,\widetilde x_{20},\epsilon) = -\int_{t}^\infty \widetilde{\mathcal M}_2^N(t,\widetilde x_{20},\epsilon))\widetilde{\mathcal M}_2^N(s,\widetilde x_{20},\epsilon)^{-1} \epsilon^N \widetilde W_{20}^N(\widetilde x_2(s,\widetilde x_{20},\epsilon), \widetilde \zeta_{2}(s,\widetilde x_{20},\epsilon),\epsilon) ds,
\end{align}
for $t\ge 0$,
see e.g. \cite[Theorem 5.3]{meiss2007a}. Here $\widetilde{\mathcal M}_2^N(t,\widetilde x_{20},\epsilon)$ is the fundamental matrix associated with the linear problem
\begin{align*}
 \epsilon \dot{\widetilde{\zeta}}_2(t) &=\left(\Omega + \epsilon \widetilde W_{21}^N(\widetilde x_2(t,\widetilde x_{20},\epsilon),\epsilon)\right) \widetilde \zeta_2(t),
\end{align*}
 with $\widetilde{\mathcal M}_2^N(0,\widetilde x_{20},\epsilon)=\operatorname{Id}$,
which by \eqref{cond} has exponential growth for $t\to \infty$. Writing
\begin{align}\eqlab{M2exp}
\widetilde{\mathcal M}_2^N(t,\widetilde x_{20},\epsilon)=:\exp({\epsilon^{-1}\Omega t}) \widehat{ \mathcal M}_2^N(t,\widetilde x_{20},\epsilon),
\end{align}we see that $\widehat{\mathcal M}_N(t,\widetilde x_{20},\epsilon)$ is the fundamental matrix associated with the regular problem:
\begin{align*}
  \dot{\widehat{\zeta}}_2(t) &=  \widetilde W_{21}^N(\widetilde x_2(t,\widetilde x_{20},\epsilon),\epsilon)\widehat \zeta_2(t),
\end{align*}
with $\widehat{\mathcal M}_2^N(0,\widetilde x_{20},\epsilon)=\operatorname{Id}$.
By assumption, we have the exponential convergence for $t\to \infty$:
\begin{align*}
 \widetilde W_{21}^N(\widetilde x_2(t,\widetilde x_{20},0),0) \to (b(1+\sigma)+2))\operatorname{Id},
\end{align*}
cf. \eqref{cond},  and from this we directly obtain
\begin{align}\eqlab{statetransition}
 \Vert \widetilde{\mathcal M}_2^N(t,\widetilde x_{20},\epsilon)\widetilde{\mathcal M}_2^N(s,\widetilde x_{20},\epsilon)^{-1}\Vert \le c_1^{-1} \e^{c_2 (t-s)},
\end{align}
see e.g. \cite[Proposition 1, p. 34]{coppel1978a},
for $c_1>0,c_2>0$ both small enough, for any $s\ge t$, $\widetilde x_2\in \widetilde{\mathcal X}_2^-$, and all $0<\epsilon\ll 1$. Here $\Vert \cdot\Vert$ denotes the operator norm. To obtain \eqref{statetransition}, we have also used that $\Vert \exp({\epsilon^{-1}\Omega t})\Vert \equiv 1$, recall \eqref{Omegadefn}.

The existence of a unique bounded solution $t \mapsto \widetilde \zeta_{2}(t,\widetilde x_{20},\epsilon), t \ge 0$, of the fixed-point equation \eqref{fixpoint}, then proceeds completely analogously to \cite[Theorem 5.3]{meiss2007a} for any $0<\epsilon\ll 1$. The desired graph is given by $\widetilde \zeta_2=\widetilde \zeta_{2}(0,\widetilde x_{2},\epsilon)=\mathcal O(\epsilon^N)$, uniformly over $\widetilde x_2\in \widetilde{\mathcal X}_2^-$, for all $0\le \epsilon\ll 1$.
%
Next, for the smoothness with respect to $\epsilon$, we first note that the stable manifold is unique, being independent of $N$, and smooth (even analytic) for $\epsilon\in (0,\epsilon_0)$ by standard theory. The boundedness of $\widetilde \zeta_2$ for any $N$ and $0<\epsilon\ll 1$ then implies (by \eqref{tildezeta2defn}) that the invariant manifold has an asymptotic expansion with respect to $\epsilon\to 0$, uniformly with respect to $x_2\in \mathcal X_2^-$. Finally, for any $k\in \mathbb N$ there is an $N\in \mathbb N$ ($N\ge 2k+1$ will do) such that the right hand side of \eqref{fixpoint} is $C^k$-smooth with respect to $\epsilon\in [0,\epsilon_0)$; notice that we lose two powers of $\epsilon$ upon differentiation due to \eqref{M2exp}.  It follows that the invariant manifold is $C^\infty$-smooth with respect to $\epsilon\in [0,\epsilon_0)$.
\end{proof}

\section{Blowup}\seclab{blowup}
Since the unstable and stable manifolds have the same asymptotic expansions (cf. \eqref{expansion}), we cannot determine the separation of the invariant manifolds by working in compact subsets of $x_2\in \mathbb C$. Notice that $x_2=\mathcal O(1)$ correspond to $x=\mathcal O(\epsilon)$ as $\epsilon\to 0$ by \eqref{buhopf}. We therefore wish to extend the unstable and stable manifolds to $x=\mathcal O(1)$ and compare the manifolds with the unperturbed manifolds, recall \lemmaref{m0pm}. For this purpose, we view \eqref{x2y2z2} as a scaling chart ($\breve \epsilon=1$):
\begin{align}\eqlab{c2}
\breve \epsilon=1:\quad 
 \begin{cases}
  x = r_2 x_2\\
  y = r_2^2 y_2,\\
  z =r_2^2 z_2,\\
  \epsilon = r_2,
 \end{cases}
\end{align}
associated with the blowup transformation
\begin{align}
 r\ge 0,\,(\breve x,\breve y,\breve z,\breve \epsilon)\in\mathbb S^3 \mapsto \begin{cases}
                                           x = r\breve x\\
                                           y =r^2\breve y\\
                                           z =r^2\breve z,\\
                                           \epsilon =r \breve \epsilon
                                          \end{cases}.\eqlab{blowup}
\end{align}
We will therefore also augment $\dot \epsilon=0$ to \eqref{model}. 
%
%
To perform the extension, we use the $\breve x=1$-chart with associated coordinates $(r_1,y_1,z_1,\epsilon_1)$:
\begin{align}
\breve x=1:\quad 
\begin{cases}
                                           x = r_1,\\
                                           y =r_1^2 y_1,\\
                                           z =r_1^2z_1,\\
                                           \epsilon =r_1 \epsilon_1.
                                          \end{cases}\eqlab{c1}
                                          \end{align}
                                          The change of coordinates between $(x_2,y_2,z_2,r_2)$, see \eqref{c2}, and $(r_1,y_1,z_1,\epsilon_1)$, see \eqref{c1}, is well-defined for $x_2\ne 0$, $\epsilon_1\ne 0$, and given by
                                          \begin{align}\eqlab{cc}
                                           \begin{cases} r_1 = r_2 x_2,\\
                                           \epsilon_1=x_2^{-1},\\
                                           y_1 =x_2^{-2} y_2,\\
                                           z_1 =x_2^{-2} z_2.
                                           \end{cases}
                                          \end{align}
The advantage of working with \eqref{c1}, and the coordinates $(r_1,y_1,z_1,\epsilon_1)$, is that both results of \lemmaref{m0pm} and \propref{unstable2} are ``visible'' within compact subsets. In particular, \lemmaref{m0pm} is visible within $\epsilon_1=0$, see also \eqref{zeta1psi}, whereas \propref{unstable2} is visible for $r_1=\mathcal O(\epsilon)$, $\epsilon_1=\mathcal O(1)$ (cf. \eqref{cc}).
                                          \subsection{Analysis in the $\breve x=1$-chart}
                                          In the $\breve x=1$-chart, we apply the change of coordinates $(r_1,y_1,z_1,\epsilon_1)\mapsto (x,y,z,\epsilon)$, defined by \eqref{c1}, to the extended vector-field given by \eqref{model} and $\dot \epsilon =0$. This leads to the following system by pull-back
                                          \begin{equation}\eqlab{r1y1z1eps10}
                                          \begin{aligned}
                                           \dot r_1 &=r_1^2\left[ 1-\epsilon_1^2  + ar_1^2 (y_1^2+ z_1^2)+r_1 F_1(r_1,r_1y_1,r_1 z_1,\epsilon_1)\right],\\
                                                                            \dot y_1 &= r_1b(-1+\epsilon_1 \sigma) y_1+z_1+ r_1 G_1(r_1,r_1 y_1,r_1 z_1,\epsilon_1) \\
                                                                            &- 2r_1 \left[ 1-\epsilon_1^2  + ar_1^2(y_1^2+ z_1^2)+ r_1 F_1(r_1,r_1 y_1,r_1 z_1,\epsilon_1)\right]y_1,\\
                                                                            \dot z_1 &= r_1b(-1+\epsilon_1 \sigma) z_1-y_1+ r_1 H_1(r_1,r_1 y_1,r_1 z_1,\epsilon_1)\\
                                                                            &- 2r_1 \left[ 1-\epsilon_1^2  + ar_1^2(y_1^2+ z_1^2)+ r_1 F_1(r_1,r_1 y_1,r_1 z_1,\epsilon_1)\right]z_1,\\
                                                                            \dot \epsilon_1 &=-r_1\epsilon_1\left[1- \epsilon_1^2  + ar_1^2 (y_1^2+ z_1^2)+ r_1 F_1(r_1,r_1 y_1,r_1 z_1,\epsilon_1)\right],
                                                        \end{aligned}
                                                        \end{equation}
                                                        where $F_1,G_1,H_1$ are defined by
                                                        \begin{align*}
                                                         W_1(r_1,r_1 y_1,r_1 z_1,\epsilon_1): = r_1^{-3} W(r_1,r_1 r_1 y_1,r_1 r_1 z_1,r_1\epsilon_1),\quad W=F,G,H,
                                                        \end{align*}
each having analytic extensions to $r_1=0$. Moreover, by \eqref{Fexpansion} it follows that
\begin{align}\eqlab{FF}
 F_1(0,0,0,\epsilon_1)\equiv \FF.
\end{align}
Notice that within the invariant set $\epsilon_1=0$, \eqref{r1y1z1eps10} reduces to \eqref{r1y1z1}, possessing invariant manifolds of the graph forms
\begin{align*}
 (y_1,z_1)= \psi_{1,0}^{\pm} (r_1),
\end{align*}
defined over the sectors $r_1\in S^{\pm}$, recall \lemmaref{m0pm}. Within $r_1=0$, the line defined by $(y_1,z_1)=(0,0)$, $\epsilon_1\ge 0$, is an elliptic critical manifold  (for $\epsilon_1>0$ it corresponds to the critical manifold of the scaling chart: $\breve \epsilon=1$, recall \eqref{layer2}). 

Since we are interested in invariant manifolds, we can multiply the right hand side of \eqref{r1y1z1eps10} by the nonzero quantity
\begin{align*}
    \frac{1-\epsilon_1^2 +r_1 \FF}{\left[1-\epsilon_1^2  + ar_1^2(y_1^2+ z_1^2)+ r_1 F_1(r_1,r_1 y_1,r_1 z_1,\epsilon_1)\right]} =: 1+ \widetilde F_1(r_1,r_1 y_1,r_1 z_1,\epsilon_1),
\end{align*}
with
\begin{align}\eqlab{quantity}
    \widetilde F_1(r_1,r_1 y_1,r_1 z_1,\epsilon_1) = \mathcal O(r_1^2).
\end{align}
Here we have used \eqref{FF}. Then \eqref{r1y1z1eps10}  becomes
\begin{equation}\eqlab{r1y1z1eps1}
  \begin{aligned}
        \dot r_1 &=r_1^2\left( 1-\epsilon_1^2  + r_1 \FF\right),\\
 \dot \zeta_1 &= \left( \Omega+r_1 W_{11}(r_1,\epsilon_1)\right)\zeta_1 +r_1W_{10}(r_1,\epsilon_1)+r_1^2 W_{12}(r_1,\zeta_1,\epsilon_1)(\zeta_1,\zeta_1),\\
\dot \epsilon_1 &=-r_1\epsilon_1\left( 1-\epsilon_1^2  + r_1 \FF\right),
\end{aligned}
\end{equation}
 upon setting $\zeta_1=(y_1,z_1)^{\mathrm{T}}$,
 for some real-analytic functions $W_{1i}$, $i\in\{0,1,2\}$. In particular, we have
\begin{align*}
W_{11}(0,\epsilon_1) = (-2-b-\epsilon_1 \sigma ) \operatorname{Id}.
\end{align*}
This follows from a simple calculation.  In the derivation of \eqref{r1y1z1eps1}, we have Taylor-expanded the right hand side of the $\zeta_1$-equation around $\zeta_1=0$. Notice, in particular, that $$W_{12}(r_1,\zeta_1,\epsilon_1)(\cdot,\cdot),$$ is bilinear.

With slight abuse of notation, we will use $\dot{(\cdot})=\frac{d}{dt}$ in the following.

We now look for (formal) invariant manifold solutions of \eqref{r1y1z1eps1}  of the graph form $\zeta_1=\zeta_1(r_1,\epsilon_1)$ (which by virtue of the multiplication by the nonzero quantity \eqref{quantity} are also (formal) invariant manifolds of \eqref{r1y1z1eps10}).
They satisfy the PDE
\begin{equation}
\eqlab{PDE}\begin{aligned}
 &r_1^2\left( 1-\epsilon_1^2  + r_1 \FF\right) \frac{\partial \zeta_1}{\partial r_1}-r_1\epsilon_1 \left( 1-\epsilon_1^2  + r_1 \FF\right)\frac{\partial \zeta_1}{\partial \epsilon_1} \\
 &=\left(\Omega +r_1W_{11}(r_1,\epsilon_1)\right) \zeta_1+r_1W_{10}(r_1,\epsilon_1)+r_1^2 W_{12}(r_1,\zeta_1,\epsilon_1)(\zeta_1,\zeta_1),
\end{aligned}
\end{equation}
in the formal sense. In particular, we will (as in \cite[Section 4.3]{uldall2024a}) use two different notions of formal series: Formal series in $r_1$ with $\epsilon_1$-depending coefficients:
\begin{align}
 \zeta_1 = \sum_{\alpha=1}^\infty m_{1,\alpha}(\epsilon_1)r_1^\alpha,\eqlab{Eseries}
\end{align}
and formal series in $\epsilon_1$ with $r_1$-depending coefficients:
\begin{align}\eqlab{wepsexp0}
 \zeta_1 = \sum_{\alpha=0}^\infty \psi_{1,\alpha}(r_1)\epsilon_1^\alpha.
\end{align}
Since \eqref{PDE} is analytic, we can again expand and understand \eqref{PDE} as an equation for formal series in both cases.
In the case \eqref{Eseries}, the coefficients $m_{1,\alpha}$ will be analytic in a neighborhood $B_\xi\subset \mathbb C$, $\xi>0$, of $\epsilon_1=0$. Here $B_\xi$ denotes the open disc in $\mathbb C$ of radius $\xi>0$ centered at the origin. On the other hand, in the case \eqref{wepsexp0} the coefficients $\psi_{1,\alpha}$ will come in pairs being analytic over the local sectors $r_1\in S^{\pm}$, respectively.


\begin{lemma}\label{lem:En1}
Fix $\xi>0$ small enough. Then there exists a unique formal series solution of \eqref{PDE} of the form \eqref{Eseries}
where each $m_{1,\alpha}:B_\xi\rightarrow \mathbb C^2$, $\alpha\in \mathbb N$, is a uniquely determined real-analytic function. Here $\xi>0$ is small enough, but independent of $\alpha$.
\end{lemma}
\begin{proof}
 We write \eqref{PDE} in the form
 \begin{align}\eqlab{PDE0}
  \zeta_1 =  \Omega^{-1} r_1 (1-\epsilon_1^2) \left(r_1\frac{\partial \zeta_1}{\partial r_1}-\epsilon_1 \frac{\partial \zeta_1}{\partial \epsilon_1}- R_1(r_1, \zeta_1,\epsilon_1)\right),
\end{align}
for some locally defined real-analytic function $R_1$.
This follows from a simple calculation. 
The result can then be proven as the proof of \cite[Proposition 2.1.]{de2020a} (by the virtue of the $r_1$-factor on the right hand side). 
\end{proof}

\begin{remark}
 Upon using the change of coordinates \eqref{cc}, \eqref{Eseries} becomes
 \begin{align*}
  \zeta_2 &=\sum_{\alpha=1}^\infty m_{1,\alpha}(x_2^{-1})x_2^{\alpha+2} r_2^\alpha,
 \end{align*}
and therefore by the uniqueness of the series, we have that 
\begin{align*}
 m_{2,\alpha}(x_2)=  m_{1,\alpha}(x_2^{-1})x_2^{\alpha{+}2},\quad x_2\ne 0,
\end{align*}
for all $\alpha\in \mathbb N$,
see \eqref{expansion}.
\end{remark}

We recall the definition of $S^{\pm}(\delta,\chi)$:
\begin{align}
 S^{\pm}(\delta,\chi)=\left\{x\in \mathbb C\,:\,0< \vert x\vert< \delta\, \mbox{ and }\, \vert \operatorname{Arg}(\pm x)\vert<  \frac{\pi+\chi}{2}\right\},\nonumber
\end{align}
with $0<\chi<\pi$ and $\delta>0$. 

\begin{lemma}\label{lem:Pnl}
Fix $\delta>0$ small enough. Then there exists two (corresponding to $\pm$) formal series solution of \eqref{PDE} of the form \eqref{wepsexp0} with $\psi_{1,\alpha}=\psi_{1,\alpha}^{\pm}$:
\begin{align}\label{eq:wepsexp}
 \zeta_1 = \sum_{\alpha=0}^\infty \psi^{\pm}_{1,\alpha}(r_1)\epsilon_1^\alpha,
\end{align}
where each $\psi^{\pm}_{1,\alpha}:S^{\pm}\rightarrow \mathbb C$, $\alpha\in\mathbb N_0$, is real-analytic on the local sector $$S^{\pm}=S^{\pm}(\delta,\chi),$$ respectively. Moreover, the two series are uniquely determined and for each fixed $\alpha\in \mathbb N_0$, $\psi^{\pm}_{1,\alpha}$ is the $1$-sum of a Gevrey-$1$ series
\begin{align*}
 \psi^{\pm}_{1,\alpha}(r_1) \sim_{1} \sum_{\beta=1}^{\infty} \psi_{1,\alpha,\beta }r_1^\beta,
\end{align*}
in the direction defined by $S^{\pm}$, respectively. In particular, the Gevrey-$1$ series on the right hand side is independent of $\pm$.


\end{lemma}
\begin{proof}
We rearrange \eqref{PDE0} to the following form
\begin{align}\eqlab{PDE00}
 r_1^2 \frac{\partial \zeta_1}{\partial r_1}-r_1\epsilon_1 \frac{\partial \zeta_1}{\partial \epsilon_1} =\Omega(1-\epsilon_1^2)^{-1} \zeta_1 + r_1 R_1(r_1, \zeta_1,\epsilon_1).
\end{align}
\rspp{With $\zeta_1$ as the formal series \eqref{wepsexp} (for simplicity we drop the $\pm$-superscript, which refers to the direction $S^{\pm}$),
we use the Faa di Bruno formula \cite[Theorem 2.1]{constantine1996a} (with $F(\zeta_1,\epsilon_1)=R_1(r_1,\zeta_1,\epsilon_1)$, $G(\epsilon_1)=(\zeta_1(\epsilon_1),\epsilon_1)$) to write
\begin{equation}\eqlab{R1expansion}
 \begin{aligned}
 R_1(r_1,\zeta_1,\epsilon_1) &= R_1(r_1,\psi_{1,0},0) + \sum_{\alpha=1}^\infty R_{1,\alpha}(r_1,\psi_{1,0},\ldots,\psi_{1,\alpha})  \epsilon_1^\alpha,
\end{aligned}
\end{equation}
where
\begin{equation}\eqlab{faa}
\begin{aligned}
 R_{1,\alpha}(r_1,\psi_{1,0},\ldots,\psi_{1,\alpha}) = &\sum_{1\le \vert \lambda \vert \le \alpha} \frac{1}{\lambda!} \frac{\partial^{\vert \lambda\vert}}{\partial \zeta_1^{(\lambda_1,\lambda_2)}\partial \epsilon_1^{\lambda_3}}R_1(r_1,\psi_{1,0},0)\times\\
 &\sum_{\gamma=1}^{\alpha} \sum_{(\sigma,\ell)\in p_{\gamma,\alpha,\lambda} } \prod_{\beta=1}^\gamma \frac{1}{\sigma_\beta! (\ell_\beta !)^{\vert \sigma_\beta\vert }} (\ell_{\beta}! \psi_{1,\ell_\beta})^{(\sigma_{1\beta},\sigma_{2\beta})} (\delta_{1\beta})^{\sigma_{3\beta}}.
\end{aligned}
\end{equation}
Here $\delta_{\alpha\beta}$ denotes Kronecker's delta.
The details of the finite index sets
\begin{align}\eqlab{p}
p_{\gamma,\alpha,\lambda}\subset \left\{(\sigma,\ell)\in (\mathbb N^3)^\gamma\times \mathbb N^\gamma\,:\,\sum_{\beta=0}^\gamma \sigma_\beta = \lambda\,\,\text{and}\,\,\sum_{\beta=0}^\gamma \vert \sigma_\beta\vert \ell_\beta=\alpha\right\}.
\end{align}
are not important, and we therefore refer to \cite[Eq. (2.2)]{constantine1996a}. We will only use that
 \begin{align*}
 R_{1,\alpha}(r_1,\psi_{1,0},\ldots,\psi_{1,\alpha}):=Q_{1,\alpha}(r_1,\psi_{1,0},\ldots,\psi_{1,\alpha-1}) + \frac{\partial}{\partial \zeta_1} R_1(r_1,\psi_{1,0},0)\psi_{1,\alpha},
 \end{align*}
 with the last term obtained from the sum in \eqref{faa} with $\lambda=(1,0,0)$ and $\lambda=(0,1,0)$ (where $p_{\gamma,\alpha,\lambda}$ are singletons), so that
$Q_{1,\alpha}$ (cf. \eqref{faa} and \eqref{p}) is a multivariate polynomial of degree $\alpha$ with respect to the variables $$(\psi_{1,1},\ldots,\psi_{1,\alpha-1}),$$ for
fixed $r_1,\psi_{1,0}$, for all $\alpha \in \mathbb N$. }
In this way, by inserting \eqref{wepsexp} into \eqref{PDE00} and collecting terms, we obtain the following equations for $\psi_{1,\alpha}$, $\alpha\in \mathbb N_0$:
\begin{align*}
 r_1^{2} \psi_{1,0}' = \Omega \psi_{1,0}+r_1 R_1(r_1,\psi_{1,0},0),
\end{align*}
and
\begin{equation}\eqlab{phin1eqn}
\begin{aligned}
 r_1^{2} \psi_{1,\alpha}' &= \Omega \psi_{1,\alpha} +r_1\left( \frac{\partial R_1}{\partial \zeta_1}(r_1,\psi_{1,0},0) +\alpha \operatorname{Id}\right)\psi_{1,\alpha} \\
 &+\sum_{\gamma=0}^{\alpha-1} \Omega A_{1,\alpha-\gamma} \psi_{1,\gamma}+r_1 Q_{1,\alpha-1}(r_1,\psi_{1,0},\ldots,\psi_{1,\alpha-1})\quad \forall\,\alpha\in \mathbb N,
\end{aligned}
\end{equation}
where
\begin{align*}
A_{1,\alpha} = \begin{cases}
                                                                                                 1 & \mbox{$\alpha$ even},\\
                           0 & \mbox{$\alpha$ odd}.
                           \end{cases}
\end{align*}
It follows that $\psi_{1,\alpha}\in S^{\pm}(\delta_\alpha,\chi)$ for $\delta_\alpha>0$ sufficiently small for all $\alpha\in \mathbb N_0$ from \cite{bonckaert2008a} (as in the proof of \lemmaref{m0pm}). In particular, $\psi_{1,\alpha}^\pm$ are the $1$-sums of a Gevrey-$1$ series in the directions defined by $S^\pm$, respectively. To obtain that $\delta_\alpha>0$ is uniformly bounded from below, we use that \eqref{phin1eqn} is a linear first order equation for $\psi_{1,\alpha}$, $\alpha\in \mathbb N$, which is regular for $r_1\ne 0$. Therefore, by \lemmaref{m0pm} ($\alpha=0$ case), the existence and uniqueness of solutions of linear equations and induction on $\alpha$, we conclude that $\psi_{1,\alpha}$ can be extended to
$S^{\pm}(\delta,\chi)$, $\delta=\delta_0$, for all $\alpha\in \mathbb N_0$.
\end{proof}

We will now use the formal invariant solutions to bring \eqref{r1y1z1eps1} into a suitable normal form, where the unstable and stable manifolds from the scaling chart can be extended to $r_1=\mathcal O(1)$ through application of the forward flow. 

In the following, we let $\mathcal G(S)\{\zeta_1,\epsilon_1\}$ denote the set of analytic functions $W_1:S\times B_{\xi}^3 \rightarrow \mathbb C^\gamma$, where $S=S^{\pm}(\delta,\chi)$ is an open sector
and where $B_\xi^3\subset \mathbb C^3$ is the open ball in $\mathbb C^3$ of radius $\xi>0$ centered at the origin, such that
\begin{align*}
 W_1(r_1,\zeta_1,\epsilon_1) = \sum_{\alpha,\beta} W_{1,\alpha,\beta}(r_1)\zeta_1^\alpha\epsilon_1^\beta ,
\end{align*}
with $\alpha,\beta$ ranging over $\mathbb N_0^2, \mathbb N_0$, respectively. Here the series is assumed to be absolutely convergent for $(\zeta_1,\epsilon_1)\in B_\xi^3\subset \mathbb C^3$, uniformly with respect $r_1\in S$, and where each $W_{1,\alpha,\beta}$ is the $1$-sum on $S$ (being continuous on the closure of $S$ with $W_{1,\alpha,\beta}(0)=0$ for all $\alpha\in \mathbb N_0^2,\beta\in \mathbb N_0$) of a Gevrey-$1$ series as $r_1\rightarrow 0$. As indicated, we suppress $\gamma$, which should be clear from the context. $\mathcal G(S)\{\epsilon_1\}$ (as a set of functions that are independent of $\zeta_1$) is defined similarly.

In the following, we fix $\delta>0$ and $\xi>0$ small enough and set $S^{\pm} :=S^{\pm}(\delta,\chi)$ with $0<\chi<\pi$. Recall that $\epsilon=r_1\epsilon_1$.
\begin{proposition}\label{prop:Hn}
 Fix any $N\in \mathbb N$.
 Then there exist two functions $W_1^{\pm,N}\in \mathcal G(S^{\pm})\{\epsilon_1\}$ such that the (blowup) transformations
 \begin{align}\eqlab{transform1}
 (r_1,\widetilde \zeta_1,\epsilon_1)\mapsto (r_1, \zeta_1,\epsilon_1), \quad r_1 \in S^{\pm},\quad (\widetilde \zeta_1,\epsilon_1) \in B_\xi^3,\end{align}
  defined by 
 \begin{align}
\zeta_1=W_1^{\pm,N}(r_1,\epsilon_1) + (r_1\epsilon_1)^N \widetilde \zeta_1,\eqlab{tildezeta1}
 \end{align}
bring \eqref{r1y1z1eps1} into the following two sets of equations by pull-back
\begin{equation}\eqlab{K1eqnsnf}
\begin{aligned}
 \dot r_1&= r_1^2\left( 1-\epsilon_1^2  + r_1 \FF\right) ,\\
 \dot{\widetilde \zeta}_1 &=\left(\Omega +r_1 \widetilde W_{11}^{\pm,N}(r_1,\epsilon_1)\right)\widetilde \zeta_1+(r_1\epsilon_1)^N \widetilde W_{10}^{\pm,N}(r_1,\widetilde \zeta_1,\epsilon_1),\\
  \dot{\epsilon}_1 &=-r_1\epsilon_1 \left( 1-\epsilon_1^2  + r_1 \FF\right),
\end{aligned}
\end{equation}
for $r_1\in S^{\pm}$, $(\widetilde \zeta_1,\epsilon_1) \in B_\xi^3$.
Here $\widetilde W_{11}^{\pm,N}\in \mathcal G(S^{\pm})\{\epsilon_1\}$, $\widetilde W_{10}^{\pm,N}\in \mathcal G(S^{\pm})\{\widetilde \zeta_1,\epsilon_1\}$ and
 \begin{align}
  W_1^{\pm,N} (r_1,0) = \psi_{1,0}^{\pm}(r_1),\quad W_1^{\pm,N} (0,\epsilon_1) = 0, \eqlab{HN0}
 \end{align}
for all $r_1\in S^{\pm}$, $\epsilon_1\in B_\xi$, respectively. Moreover,
\begin{align}\eqlab{T11expansion}
 \widetilde W_{11}^{\pm,N}(0,\epsilon_1) = (-2-b+\epsilon_1 \sigma) \operatorname{Id}.
%
\end{align}
\end{proposition}
\begin{remark}
 Notice that \eqref{transform1} is of blowup-type since it reduces to $\zeta_1=W_1^{\pm,N}(r_1,\epsilon_1)$ for any $(r_1,\epsilon_1)\,:\,r_1\epsilon_1=0$.
\end{remark}

\begin{proof}
 We follow \cite[Lemma 4.5]{uldall2024a} and take
 \begin{align*}
  W_1^{\pm,N}(r_1,\epsilon_1):=\sum_{\alpha=0}^{2N-1} \psi_{1,\alpha}^{\pm}(r_1)\epsilon_1^\alpha + \sum_{\alpha=1}^{2N-1}( m_{1,\alpha}-J^{2N-1}(m_{1,\alpha}))(\epsilon_1) r_1^\alpha,
 \end{align*}
where $J^\gamma(W_1)$, $\gamma\in \mathbb N$, denotes the $\gamma$th order jet/partial sum of an analytic function $W_1=\sum_{\alpha=0}^\infty W_{1,\alpha} \epsilon_1^\alpha$:
 \begin{align*}
  J^\gamma(W_1)(\epsilon_1): = \sum_{\alpha=0}^\gamma W_{1,\alpha} \epsilon_1^\alpha.
 \end{align*}
Clearly, $W_1^{\pm,N}\in \mathcal G(S^{\pm})\{\epsilon_1\}$, satisfying \eqref{HN0}. 
 Then by construction $\zeta_{1}=W_1^{\pm,N}(r_1,\epsilon_1)$ defines an invariant manifold up to error-terms of order $\mathcal O((r_1\epsilon_1)^{2N})=\mathcal O(\epsilon^{2N})$, in the sense that there exists two locally defined real-analytic functions $\widetilde R_1^{\pm,N}\in \mathcal G(S^{\pm})\{\epsilon_1\}$:
 \begin{align*}
   &r_1^2\left( 1-\epsilon_1^2  + r_1 \FF\right) \frac{\partial W_1^{\pm,N}}{\partial r_1}-r_1\epsilon_1 \left( 1-\epsilon_1^2  + r_1 \FF\right)\frac{\partial W_1^{\pm,N}}{\partial \epsilon_1} \\
   &-\left(\left(\Omega +r_1 W_{11}\right) W_1^{\pm,N}+r_1W_{10}+r_1^2 W_{12}(W_1^{\pm,N})^2\right)\\
   &=:(r_1\epsilon_1)^{2N} \widetilde R_1^{\pm,N}(r_1,\epsilon_1)\quad \forall\,r_1\in S^{\pm},\,\epsilon_1\in B_\xi,
 \end{align*}
 writing (for simplicity) $$W_{10}(r_1,\epsilon_1), \,\, W_{11}(r_1,\epsilon_1), \,\, W_{12}(r_1,W_1^{\pm,N},\epsilon_1)(W_1^{\pm,N},W_1^{\pm,N}),$$
 compactly as $$W_{10}, \,\, W_{11},\,\, W_{12}(W_1^{\pm,N})^2,\,\, \mbox{respectively}.$$
 Therefore \eqref{tildezeta1} gives
 \begin{align*}
  (r_1\epsilon_1)^N \dot{\widetilde \zeta}_1 &=\left(\Omega +r_1 W_{11}\right) \zeta_1+r_1W_{10}+r_1^2 W_{12}(\zeta_1)^2\\
  &-r_1^2 \left( 1-\epsilon_1^2  + r_1 \FF\right)\frac{\partial W_1^{\pm,N}}{\partial r_1} \\
  &+ r_1\epsilon_1 \left( 1-\epsilon_1^2  + r_1 \FF\right) \frac{\partial W_1^{\pm,N}}{\partial \epsilon_1}\\
  &=\left(\Omega +r_1 W_{11}\right)(r_1\epsilon_1)^N \widetilde \zeta_1 -(r_1\epsilon_1)^{2N} \widetilde R_1^{\pm,N}\\
  &+r_1^2 \left(W_{12}\left(W_1^{\pm,N}+(r_1\epsilon_1)^N \widetilde \zeta_1)\right)^2-W_{12}\left(W_1^{\pm,N}\right)^2\right)\\
  &=:\left(\Omega +r_1\widetilde W_{11}^{\pm,N}(r_1,\epsilon_1)\right) (r_1\epsilon_1)^{N} \widetilde \zeta_1 + (r_1\epsilon_1)^{2N} \widetilde W_{10}^{\pm,N}(r_1,\widetilde \zeta_1,\epsilon_1),
 \end{align*}
 where we have used a Taylor-expansion with respect to $\widetilde \zeta_1$ in the final equality. 
The expansion of $\widetilde W_{11}^{\pm,N}$ follows from a simple calculation using that $W_1^{\pm,N}(r_1,\epsilon_1)=\mathcal O(r_1)$.
\end{proof}

Below we will use the flow of \eqref{K1eqnsnf}, $r_1\in S^{\pm}$, to extend the invariant manifolds from $x=r_1=\mathcal O(\epsilon)$ (recall \propref{unstable2}) to $x=r_1=\mathcal O(1)$. We focus on the unstable manifold and the case corresponding to ``$+$''. The case ``$-$'' can be handled in the same way and the details of this case is therefore left out.
\subsection{Elliptic paths}

 In this section, we first focus on the $(r_1,\epsilon_1)$-subsystem of \eqref{K1eqnsnf}:
\begin{equation}\eqlab{r1eps1}
\begin{aligned}
 \dot r_1 &= r_1^2 (1-\epsilon_1^2+r_1 \FF),\\
  \dot \epsilon_1 &=-r_1\epsilon_1 (1-\epsilon_1^2+r_1 \FF),
\end{aligned}
\end{equation}
with $\dot{(\cdot})=\frac{d}{dt}$,
which by construction decouples.  

We call trajectories $(r_1(t),\epsilon_1(t)))$, $t\in I\subset \mathbb R$, \textit{elliptic paths} (see e.g. \cite{hayes2016a}) since the $\widetilde \zeta_1$-dynamics, see \eqref{K1eqnsnf}, is dominated by oscillations along such paths, see also \lemmaref{Phi} below. In contrast to \textit{hyperbolic paths}, see \secref{difference} below, where the $\widetilde \zeta_1$-dynamics is dominated by contraction and expansion, the elliptic paths offer a suitable control of the $\widetilde \zeta_1$-dynamics to extend the unstable manifold, see also \cite{neishtadt1987a,neishtadt1988a}.

Notice that upon restricting attention to $\epsilon=r_1\epsilon_1=\textrm{const}.$, \eqref{r1eps1} can be reduced to
\begin{align}
 \dot r_1 &=r_1^2 -\epsilon^2 + r_1^3 \FF.\eqlab{r1eqnel}
\end{align}
We focus on the closed region $\mathcal A^+(\epsilon)$ defined by $r_1\in S^+(\delta,\chi)$ and $\vert r_1\vert \in [\xi^{-1}\epsilon,\delta]$, $\xi>0$ and $\delta>0$ small enough, and all $0<\epsilon\ll 1$. In polar coordinates $r_1=\rho_1 e^{i\phi}$, the region becomes $\rho_1\in [\xi^{-1}\epsilon,\delta]$, $\phi\in (-\frac{\pi+\chi}{2},\frac{\pi+\chi}{2})$. Moreover, in these coordinates \eqref{r1eqnel} becomes
\begin{equation}\eqlab{rho1phieqn}
\begin{aligned}
 \dot \rho_1 & = \rho_1^2 \left(\cos \phi-\rho_1^{-2}\epsilon^2 \cos \phi+\rho_1 \cos(2\phi) \FF\right),\\
 \dot \phi &=\sin \phi  \left(1+\rho_1^{-2} \epsilon^2+2\rho_1 \FF \cos \phi\right)\rho_1.
\end{aligned}
\end{equation}
Clearly, $\dot \phi\gtrless 0$ for $\phi\gtrless 0$ whereas $\dot \rho_1=0$ for $\rho_1\in [\xi^{-1}\epsilon,\delta]$ gives two solutions
\begin{align}\eqlab{phipm}
\phi= \phi_\pm(\rho_1,\epsilon^2\rho_1^{-2}):=\pm \frac{\pi}{2}+\mathcal O(\rho_1,\epsilon^2\rho_1^{-2}), \quad \rho_1\in [\xi^{-1}\epsilon,\delta],
\end{align} 
within $\mathcal A^+(\epsilon)$. This latter property can be obtained from the implicit function theorem.
Within the subset of $\mathcal A^+(\epsilon)$ defined by $$\phi\in (\phi_-(\rho_1,\epsilon^2\rho_1^{-2}),\phi_+(\rho_1,\epsilon^2\rho_1^{-2})),\quad \rho_1\in [\xi^{-1}\epsilon,\delta],$$ we have $\rho_1'>0$; recall that $\xi>0$ and $\delta>0$ are fixed small enough. We sketch the phase-portrait in \figref{elliptic}.

\begin{lemma}\lemmalab{Tbound}
 Let $r_1(t)\in \mathcal A^+(\epsilon)$ denote a forward solution of \eqref{r1eqnel} in $\mathcal A^+(\epsilon)$ with initial condition $r_1(0)$ in the interior of $\mathcal A^+(\epsilon)$. Then for any $0<\epsilon\ll 1$ there is a $T=T(r_1(0),\epsilon)>0$ such that $r_1(t)\in \mathcal A^+(\epsilon)$ for all $t\in [0,T)$ and $r_1(T)\in \partial \mathcal A^+(\epsilon)$. Moreover, there exists a constant $C>0$ such that
 \begin{align}\eqlab{Tbound}
  T(r_1(0),\epsilon)\le C \epsilon^{-1}\quad \forall\,r_1(0)\in \mathcal A^+(\epsilon),\,0<\epsilon\ll 1.
  \end{align}
\end{lemma}
\begin{proof}
 The existence of $T$ follows from the analysis above, see also \figref{elliptic} for an illustration. To obtain the bound, we consider \eqref{rho1phieqn}. Suppose first that $\phi(t)\in [-c,c]$ with $c>0$ fixed small enough. Then
 \begin{align*}
 \dot \rho_1 >\frac12 \rho_1^2 (1-\rho_1^{-2}\epsilon^2),
 \end{align*}
 for $\xi,\delta>0$ small enough,
 from which an upper bound of $T$ of the form $C\epsilon^{-1}$ can be obtained. Suppose next that $\phi(t)\notin [-c,c]$. Then for $\xi>0$ and $\delta>0$ small enough: 
 \begin{align*}
  \vert\dot \phi\vert > \frac12 \vert \sin c\vert \rho_1\ge \frac12 \vert \sin c\vert \xi^{-1} \epsilon>0,
 \end{align*}
from which we estimate $T\le C\epsilon^{-1}$. In combination, these estimates prove that $r_1(T)\in \partial \mathcal A^+(\epsilon)$ with $0<T(r_1(0),\epsilon) \le C\epsilon^{-1}$, $C>0$ large enough, for all $0<\epsilon\ll 1$.
 \end{proof}

  \begin{figure}[h!]
\begin{center}
\includegraphics[width=.45\textwidth]{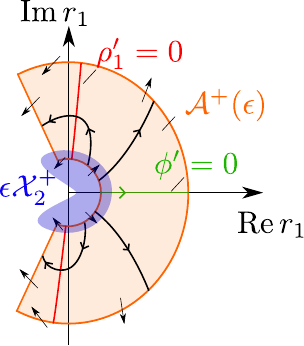}
\end{center}
\caption{Illustration of the phase portrait of \eqref{r1eqnel} on the region $\mathcal A^+(\epsilon)$ (orange). Here $r_1\in \epsilon\mathcal X_2^+$ (blue) is obtained from $x_2\in \mathcal X_2^+$ in \propref{unstable2}, cf. the change of coordinates in \eqref{cc}. We fix $\mathcal X_2^+$ so large so that $\epsilon \mathcal X_2^+$ has a non-empty intersection with $\mathcal A^+(\epsilon)$ for all $0<\epsilon\ll 1$. Then we extend the unstable manifold of $E^+(\epsilon)$ in \propref{unstable2} as a graph over $\mathcal A^+(\epsilon)$ by using the forward flow of \eqref{r1eqnel} (with initial conditions in $\epsilon \mathcal X_2^+$), see \propref{prop3}.}
\figlab{elliptic}
\end{figure}
%

 Now, let $r_1(t)\in \mathcal A^+(\epsilon)$, $t\in [0,T]$, denote a solution of \eqref{r1eqnel} with $r_1(0)\in \mathcal A^+(\epsilon)$. Here $T=T(r_1(0),\epsilon)>0$ is as in \lemmaref{Tbound}. Then we denote by $t\mapsto \widetilde{\mathcal M}_1^N(t,r_1(0),\epsilon)$, $t\in [0,T]$, the fundamental matrix associated with the linear problem
 \begin{align}
  \dot{\widetilde \zeta}_1(t) &=\left(\Omega+r_1(t)\widetilde W_{11}^{\pm,N}(r_1(t),\epsilon r_1(t)^{-1})\right)\widetilde \zeta_1(t),\quad t\in [0,T],\eqlab{lineareqn}
  \end{align}
with $\widetilde{\mathcal M}_1^N(0,r_1(0),\epsilon)=\operatorname{Id}$. Here we have used that $\epsilon_1=\epsilon r_1^{-1}$. In the following, we fix $\xi>0,\delta>0$ small enough.
\begin{lemma}\lemmalab{Phi}

The fundamental matrix $\widetilde{\mathcal M}_1^N(t,r_1(0),\epsilon)$ has the following expansion
  \begin{align}\eqlab{M1Nexpansion}
   \widetilde{\mathcal M}_1^N(t,r_1(0),\epsilon) = \left(\frac{r_1(0)^2-\epsilon^2}{r_1(t)^2-\epsilon^2}\right)^{\frac{2+b}{2}} \exp(\Omega t)\mathcal O(1),\quad t\in [0,T].
  \end{align}
  In particular,
  \begin{align}
  \Vert \widetilde{\mathcal M}_1^N(t,r_1(0),\epsilon) \widetilde{\mathcal M}_1^N(s,r_1(0),\epsilon)^{-1}\Vert \le
C\epsilon^{-2-b},\eqlab{Phib}
\end{align}
for some (new) constant $C>0$,
for all $0< \epsilon\ll 1,\,0\le s\le t\le T$.
\end{lemma}
\begin{proof}
We write
\begin{align}\eqlab{widetildeM}
 \widetilde{\mathcal M}_1^N(t,r_1(0),\epsilon) = \left(\frac{r_1(0)^2-\epsilon^2}{r_1(t)^2-\epsilon^2}\right)^{\frac{2+b}{2}} \e^{(\Omega +\sigma \epsilon \operatorname{Id}) t}\widehat{\mathcal M}_1^N(t,r_1(0),\epsilon).
\end{align}
A simple calculation then shows that $\widehat{\mathcal M}_1^N(t,r_1(0),\epsilon)$ is the fundamental matrix associated with the linear problem
\begin{align*}
\dot{\widehat \zeta}_1(t) &= \mathcal O(r_1(t)^2)\widehat \zeta_1,
\end{align*}
satisfying $\widehat{\mathcal M}_1^N(0,r_1(0),\epsilon)=\operatorname{Id}$. Here we have used \eqref{T11expansion}, $\xi>0$ and $\delta>0$ small enough, and that $\vert r_1(t)\vert\in [\xi^{-1}\epsilon,\delta]$, $t\in [0,T]$, by assumption. But then
\begin{align*}
\frac{d\widehat \zeta_1}{dr_1} =  \mathcal O(1) \widehat \zeta_1,
\end{align*}
by the chain rule, recall \eqref{r1eqnel}. It follows that $\Vert \widehat{\mathcal M}_1^N(t,r_1(0),\epsilon)\widehat{\mathcal M}_1^N(s,r_1(0),\epsilon)^{-1}\Vert$, $0\le s\le t\le T$, is uniformly bounded. This proves \eqref{M1Nexpansion} and \eqref{Phib},
upon
using \eqref{Tbound} and $\Vert \exp({\Omega t})\Vert\equiv 1$.
\end{proof}

\subsection{Extension of the unstable manifold}
By taking the compact set $\mathcal X_2^+$ in \propref{unstable2} large enough, we can by the change of coordinates \eqref{cc}, write the unstable manifold of $E^+(\epsilon)$ in the graph form
\begin{align}
 (y_1,z_1) = m_1^+(r_1,\epsilon):=(\epsilon r_1^{-1})^2 m_2^+(\epsilon^{-1} r_1,\epsilon),\eqlab{init}
\end{align}
on the boundary $r_1\in \partial^1 \mathcal A^+(\epsilon)$ of $\mathcal A^+(\epsilon)$ defined by $\vert r_1\vert =\xi^{-1} \epsilon$. 
We then use the flow with \eqref{init} as initial conditions to extend the graph \eqref{init} over the region $r_1\in \mathcal A^+(\epsilon)$.

Notice that since $m_2^+$ is asymptotic to the series \eqref{expansion}, and the series \eqref{Eseries} coincides with \eqref{expansion} upon the change of coordinates \eqref{cc}, it follows that \eqref{init} becomes
\begin{align}\eqlab{tilde_init}
 \widetilde \zeta_1(r_1,\epsilon) = \epsilon^N \widetilde m_1^{+,N}(r_1,\epsilon),
\end{align}
with respect to the $\widetilde \zeta_1$-coordinate,
for $r_1\in \partial^1 \mathcal A^+(\epsilon)$ and some locally defined real-analytic function $\widetilde m_1^{+,N}(\cdot,\epsilon)$, recall \eqref{transform1}. We then use \lemmaref{Phi} to write \eqref{K1eqnsnf} as a fixed-point equation:
\begin{equation}\eqlab{fixed}
\begin{aligned}
 \widetilde \zeta_1(t) &= \widetilde{\mathcal M}_1^N(t,r_1(0),\epsilon) \epsilon^N \widetilde m_1^{+,N}(r_1(0),\epsilon )\\
 &+\int_0^t \widetilde{\mathcal M}_1^N(t,r_1(0),\epsilon) \widetilde{\mathcal M}_1^N(s,r_1(0),\epsilon)^{-1} \epsilon^N \widetilde W_{10}^{\pm,N}(r_1(s),\widetilde \zeta_1(s),\epsilon r_1(s)^{-1}) ds,
\end{aligned}
\end{equation}
where $r_1(t)\in \mathcal A^+(\epsilon)$ is the solution of \eqref{r1eqnel} in $\mathcal A^+(\epsilon)$ with $r_1(0)\in \partial^1\mathcal A^+(\epsilon)$. In fact, to ensure that $r_1(t)\in \mathcal A^+(\epsilon)$ for $t\in [0,T]$, we restrict attention to initial conditions $r_1(0)=\xi^{-1} \epsilon e^{i\phi}$ with $\phi\in (\phi_-,\phi_+)$, recall \eqref{phipm}. See also \figref{elliptic}. In this way, we parameterize a subset $\widetilde{\mathcal A}(\epsilon)\subset \mathcal A^+(\epsilon)$ through $t\in [0,T]$ and $\phi\in (\phi_-,\phi_+)$. The fact that $\widetilde{\mathcal A}(\epsilon)\ne \mathcal A^+(\epsilon)$ is not important. Indeed, $\widetilde{\mathcal A}(\epsilon)$ clearly contains a region (similar to $\mathcal A^+(\epsilon)$, see \figref{elliptic}) with $\vert r_1\vert \in [\widetilde \xi^{-1}\epsilon,\widetilde \delta]$, $\phi\in (-\frac{\pi+\chi}{2},\frac{\pi+\chi}{2})$, where $0<\widetilde \xi<\xi$ and $0<\widetilde \delta<\delta$, for all  $0<\epsilon\ll 1$, and we can just work with this region instead.



\begin{proposition}\proplab{prop3}
The unstable manifold of $E^+(\epsilon)$ is a graph over $r_1\in \mathcal A(\epsilon)$:
\begin{align}\eqlab{graph1}
 (y_1,z_1) = m_1^+(r_1,\epsilon),\quad r_1\in \mathcal A(\epsilon),
\end{align}
for all $0<\epsilon<\epsilon_0$, $0<\epsilon_0\ll 1$.
Here $m_1^+(\cdot,\epsilon):\mathcal A(\epsilon)\rightarrow \mathbb C^2$ is real-analytic with respect to $r_1\in \mathcal A^+(\epsilon)$ and $C^\infty$ with respect to $\epsilon\in [0,\epsilon_0)$. In particular,
\begin{align}
  m_1^+(r_1,0) = \psi_{1,0}^+(r_1), \quad r_1 \in S^+.\eqlab{m1eps0}
\end{align}

\end{proposition}
\begin{proof}
 We work in the $\widetilde \zeta_1$-coordinates and consider the fixed-point formulation \eqref{fixed}. It suffices to estimate the solution. Fix $C>0$ and let $T_0\le T(r_1(0),\epsilon)$ be so that  $\vert \widetilde \zeta_1(t)\vert\le C$ for all $t\in [0,T_0]$. Then by \eqref{Phib}, we conclude from \eqref{fixed} that
\begin{align}
 \vert \widetilde \zeta_1(t)\vert \le \mathcal O(1) \epsilon^{N-2-b} + \mathcal O(1) \epsilon^{N-2-b} T_0\le \mathcal O(1)\epsilon^{N-3-b},\eqlab{tildezeta1est}
\end{align}
for all $t\in [0,T_0]$, upon using \eqref{Tbound}, and that $\epsilon T$ is uniformly bounded by \lemmaref{Tbound}, see \eqref{Tbound}. We therefore take $N>3+b$ and $T_0=T$ for all $0<\epsilon\ll 1$. This gives the existence of the graph \eqref{graph1} for all $0<\epsilon<\epsilon_0$, $0<\epsilon_0\ll 1$. Notice in particular that \eqref{m1eps0} follows from \eqref{tildezeta1} and \eqref{tildezeta1est}, see also \eqref{HN0}. Finally, for the smoothness with respect to $\epsilon$, we can proceed as in the proof of \propref{unstable2}.
\end{proof}

\section{A dynamical systems approach to the difference }\seclab{difference}
It now follows from \propref{unstable2} and \propref{prop3}, see also \eqref{c1} and \eqref{c2}, that we have invariant manifolds of \eqref{model} as graphs
\begin{align}
    (y,z)&=m^{\pm}(x,\epsilon),\quad x\in S^{\pm},\eqlab{finalmpm}
\end{align}
respectively. (Although \propref{prop3} only relates to the unstable manifold of $E^+$, we can obtain a similar result for the stable manifold of $E^-$ upon time-reversal.) By taking the sets $\mathcal X^{\pm}$ large enough, we can replace $S^{\pm}$ by their closures, $\operatorname{Cl}(S^{\pm})$, in \eqref{finalmpm}. We then consider the difference
\begin{align*}
 \Delta m:= m^{+}(x,\epsilon)-m^{-}(x,\epsilon),\quad \Delta m =:( \Delta y,
  \Delta z),
\end{align*}
on the common domain $x\in S^+\cap S^-$, which by using the closures of $S^{\pm}$, includes a subset $i[-\delta, \delta]$, $\delta>0$ small enough, of the imaginary axis $i \mathbb R\subset \mathbb C$.
Each $m^{\pm}$ is third order in the following sense
\begin{align*}m^{\pm}(x,\epsilon) = \mathcal O(\vert (x,\epsilon)\vert^3),
\end{align*}
for $(x,\epsilon)\to (0,0)$ in $S^{\pm}\times [0,\epsilon_0)$, respectively. This follows from the analysis in the charts.

Now, $(y,z)=m^{+}(x,\epsilon)$, $x\in S^+$, and $(y,z)=m^{-}(x,\epsilon)$, $x\in S^-$, are invariant manifolds, satisfying the invariance equation
\begin{align*}
 \frac{d}{dx} \begin{pmatrix}
               y\\
               z
              \end{pmatrix}
  = J(x,y,z,\epsilon),
\end{align*}
with
\begin{align*}
 J(x,y,z,\epsilon): = \frac{1}{x^2-\epsilon^2+a(y^2+z^2) +F }\left(\left(\Omega +b(\sigma \epsilon-x)\operatorname{Id}\right)\begin{pmatrix}
                                                                                                                              y\\
                                                                                                                              z
                                                                                                                             \end{pmatrix}+\begin{pmatrix}
G\\
H
\end{pmatrix}\right),
\end{align*}
where $W=W(x,y,z,\epsilon)$, $W=F,G,H$. In particular, we note that $J(x,m^{\pm}(x,\epsilon),\epsilon)$ are both well-defined on $x\in S^+\cap S^-$ for all $0<\epsilon\ll 1$. Consequently,
\begin{equation}\eqlab{DeltayDeltaz}\begin{aligned}
 \frac{d}{dx}\Delta m &= J(x,m^-+\Delta m,\epsilon)-J(x,m^-,\epsilon)\\
 &=:\begin{pmatrix}
     A_{11}(x,\epsilon) & A_{12}(x,\epsilon)\\
     A_{21}(x,\epsilon) & A_{22}(x,\epsilon)
    \end{pmatrix}
\Delta m,
\end{aligned}
\end{equation}
%
by the mean value theorem, where
\begin{align*}
    A_{11}(x,\epsilon)& = \frac{b(-x+\epsilon \sigma) +G'_y(x,y_*,z_*,\epsilon)}{x^2-\epsilon^2+a(y_*^2+z_*^2)+F(x,y_*,z_*,\epsilon)}\\
    &-
     \frac{b(-x+\epsilon \sigma)y_*+z_* +G(x,y_*,z_*,\epsilon)}{(x^2-\epsilon^2+a(y_*^2+z_*^2)+F(x,y_*,z_*,\epsilon))^2}(2a y_*+F'_y(x,y_*,z_*,\epsilon),\\
         A_{12}(x,\epsilon)& = \frac{1+G'_z(x,y_*,z_*,\epsilon)}{x^2-\epsilon^2+a(y_*^2+z_*^2)+F(x,y_*,z_*,\epsilon)}\\
    &-
     \frac{b(-x+\epsilon \sigma)y_*+z_* +G(x,y_*,z_*,\epsilon)}{(x^2-\epsilon^2+a(y_*^2+z_*^2)+F(x,y_*,z_*,\epsilon))^2}(2a z_*+F'_z(x,y_*,z_*,\epsilon),
     \end{align*}
     and
     \begin{align*}
    A_{21}(x,\epsilon)& = \frac{-1+H'_y(x,y_*,z_*,\epsilon)}{x^2-\epsilon^2+a(y_*^2+z_*^2)+F(x,y_*,z_*,\epsilon)}\\
    &-
     \frac{b(-x+\epsilon \sigma)z_*-y_* +H(x,y_*,z_*,\epsilon)}{(x^2-\epsilon^2+a(y_*^2+z_*^2)+F(x,y_*,z_*,\epsilon))^2}(2a y_*+F'_y(x,y_*,z_*,\epsilon),\\
    A_{22}(x,\epsilon)& = \frac{b(-x+\epsilon \sigma) +H'_z(x,y_*,z_*,\epsilon)}{x^2-\epsilon^2+a(y_*^2+z_*^2)+F(x,y_*,z_*,\epsilon)}\\
    &-
     \frac{b(-x+\epsilon \sigma)z_*-y_* +H(x,y_*,z_*,\epsilon)}{(x^2-\epsilon^2+a(y_*^2+z_*^2)+F(x,y_*,z_*,\epsilon))^2}(2a z_*+F'_z(x,y_*,z_*,\epsilon).
\end{align*}
Here \begin{align*}(y_*,z_*)=(y_*,z_*)(x,\epsilon) = \mathcal O(\vert (x,\epsilon)\vert^3),
\end{align*}
for $(x,\epsilon)\to (0,0)$ in $(S^+\cap S^-)\times [0,\epsilon_0)$, are (four different pairs of) known real-analytic functions (the intermediate points in the application of the mean value theorem), depending smoothly on $\epsilon\in [0,\epsilon_0)$ (in the same sense as $m^\pm$ in the charts above).
It will not be important to distinguish the pairs, so we use the same symbols for all of them. Notice by \eqref{Fexpansion}:
\begin{align}\eqlab{Fexpansionst}
 F(x,y_*,z_*,\epsilon) = \FF x^3+\mathcal O(\vert (x,\epsilon)\vert^4),
\end{align}
for $(x,\epsilon)\to (0,0)$ in $(S^+\cap S^-)\times [0,\epsilon_0)$.

Now, the quantities $A_{\alpha\beta}$ are each of the form ``$0/0$'' for $(x,\epsilon)=(0,0)$, so we study the resulting system using blowup. It will suffice to work along the imaginary axis, so we put $$x=iv,$$ and use the real blowup:
\begin{align}\eqlab{blowupfinal}
    \rho \ge 0,\,(\breve \rho,\breve \epsilon)\in S^1\mapsto \begin{cases}
v=\rho \breve v,\\
\epsilon = \rho \breve \epsilon.
    \end{cases}
\end{align}
\begin{remark}
    Notice that in comparison with the earlier blowup, we only blowup $x$ and $\epsilon$ and we focus on $x\in i\mathbb R$. Therefore \eqref{blowupfinal} is real. Notice also, that we abuse notation slightly in \eqref{blowupfinal} by redefining $\breve \epsilon$ (which was used in a different (complex) context in \eqref{blowup}).
    
    Moreover, as \eqref{model} is real-analytic, so that $(x(t),y(t),z(t))$, $t\in \mathbb R$, is a solution if and only if $(\overline x(t),\overline y(t),\overline z(t))$, $t\in \mathbb R$, is a solution, we have
    \begin{align}
        {A_{\alpha\beta}(-iv,\epsilon)} = \overline{A_{\alpha\beta}}(iv,\epsilon),\eqlab{Eprop}
    \end{align}
    and in particular
    \begin{align*}
        \begin{cases}
            \Delta y(-iv) =\overline{\Delta y}(iv),\\
            \Delta z(-iv) = \overline{\Delta z}(iv),
        \end{cases}
    \end{align*}
    where the bar denotes complex conjugation, for all $v\in [-\delta,\delta]$. The system
    \begin{equation}\eqlab{DeltayDeltaz1}
\begin{aligned}
   \frac{d\Delta y}{dv} &=i\left(A_{11}(iv,\epsilon) \Delta y+A_{12}(iv,\epsilon) \Delta z\right),\\
   \frac{d\Delta z}{dv} &=i\left(A_{21}(iv,\epsilon) \Delta y+A_{22}(iv,\epsilon) \Delta z\right),
\end{aligned}
\end{equation}
    obtained by substituting $x=iv$ into \eqref{DeltayDeltaz}, is therefore also invariant with respect to the time-reversal symmetry
    \begin{align}\eqlab{symmetry}
        (v,\Delta y,\Delta z)\mapsto (-v,\overline{\Delta y},\overline{\Delta z}).
    \end{align}
\end{remark}
We work in the charts $\breve v=1$, $\breve \epsilon=1$ and $\breve v=-1$ with chart-specific coordinates $(\rho_1,\epsilon_1)$, $(\rho_2,v_2)$ and $(\rho_3,\epsilon_3)$, respectively, defined by:
\begin{align*}
    \breve v=1:\quad \begin{cases}
        v =\rho_1,\\
        \epsilon =\rho_1\epsilon_1,
    \end{cases}\\
    \breve \epsilon=1:\quad \begin{cases}
        v =\rho_2v_2,\\
        \epsilon =\rho_2,
    \end{cases}\\
    \breve v=-1:\quad \begin{cases}
        v =-\rho_3,\\
        \epsilon =\rho_3\epsilon_3.
    \end{cases}
\end{align*}
The change of coordinates are given by the expressions
\begin{align}\eqlab{cc2}
    \begin{cases}
       \rho_1 =\rho_2 v_2,\\
        \epsilon_1 =v_2^{-1},
    \end{cases}\qquad \begin{cases}
       \rho_3 =\rho_2 (-v_2),\\
        \epsilon_3 =(-v_2)^{-1}.
    \end{cases}
\end{align}
In the scaling chart: $\breve \epsilon=1$, we find
\begin{equation}\eqlab{Aalphabeta2}
\begin{aligned}
    A_{11}(\epsilon i v_2,\epsilon) &= \epsilon^{-2} \left(\frac{b\epsilon (-i v_2+\sigma)+\mathcal O(\epsilon^2) }{-v_2^2-1-\epsilon \FF iv_2^3+\mathcal O(\epsilon^2)}-\frac{\mathcal O(\epsilon^3)}{\left(-v_2^2-1-\epsilon \FF iv_2^3+\mathcal O(\epsilon^2)\right)^2}\right),\\
    A_{12}(\epsilon i v_2,\epsilon) &= \epsilon^{-2} \left(\frac{1+\mathcal O(\epsilon^2) }{-v_2^2-1-\epsilon \FF iv_2^3+\mathcal O(\epsilon^2)}-\frac{\mathcal O(\epsilon^3)}{\left(-v_2^2-1-\epsilon \FF iv_2
    ^3+\mathcal O(\epsilon^2)\right)^2}\right),\\
    A_{21}(\epsilon i v_2,\epsilon)&=\epsilon^{-2} \left(\frac{-1+\mathcal O(\epsilon^2) }{-v_2^2-1-\epsilon \FF iv_2^3+\mathcal O(\epsilon^2)}-\frac{\mathcal O(\epsilon^3)}{\left(-v_2^2-1-\epsilon \FF iv_2^3+\mathcal O(\epsilon^2)\right)^2}\right),\\
    A_{22}(\epsilon  i
    v_2,\epsilon)&=\epsilon^{-2} \left(\frac{b\epsilon (-i v_2+\sigma)+\mathcal O(\epsilon^2) }{-v_2^2-1-\epsilon \FF iv_2^3+\mathcal O(\epsilon^2)}-\frac{\mathcal O(\epsilon^3)}{\left(-v_2^2-1-\epsilon \FF iv_2^3+\mathcal O(\epsilon^2)\right)^2}\right).
\end{aligned}
\end{equation}
Here we have used \eqref{Fexpansionst}.
This leads to the following system 
\begin{align*}
  \frac{d\Delta y}{dv_2} &= \frac{i}{\epsilon(-v_2^2-1-\epsilon \FF iv_2^3)}\left( (b\epsilon(-iv_2+\sigma)+\mathcal O(\epsilon^2)) \Delta y+ (1+\mathcal O(\epsilon^2))\Delta z \right),\\
 \frac{d\Delta z}{dv_2} &= \frac{i}{\epsilon(-v_2^2-1-\epsilon \FF iv_2^3))}\left( (b\epsilon(-iv_2+\sigma)+\mathcal O(\epsilon^2)) \Delta z- (1+\mathcal O(\epsilon^2))\Delta y \right),
\end{align*}
for $v_2\in I$, with $I\subset \mathbb R$ compact. To study this system, we use the following equivalent first order formulation:
\begin{equation}\eqlab{Deltayz2}
\begin{aligned}
v_2' &=-\epsilon ,\\
\Delta y' &= \frac{i}{v_2^2+1+\epsilon \FF iv_2^3}\left( (b\epsilon(-iv_2+\sigma)+\mathcal O(\epsilon^2)) \Delta y+ (1+\mathcal O(\epsilon^2))\Delta z \right),\\
 \Delta z' &= \frac{i}{v_2^2+1+\epsilon \FF iv_2^3}\left( (b\epsilon(-iv_2+\sigma)+\mathcal O(\epsilon^2)) \Delta z- (1+\mathcal O(\epsilon^2))\Delta y \right).
\end{aligned}
\end{equation}
The advantage of working with \eqref{Deltayz2} is that we can use dynamical systems technique. In particular,  $\epsilon=0$ is well-defined:
\begin{align*}
v_2'&=0,\\
 \Delta y' &=\frac{i}{v_2^2+1} \Delta z,\\
 \Delta z' &=-\frac{i}{v_2^2+1}\Delta y.
\end{align*}
We notice that $(v_2,0,0)$ is a normally hyperbolic line, the linearization having eigenvalues $0,\pm 1$. (In this sense, the imaginary axis is a \textit{hyperbolic path}, see \cite{hayes2016a,new}.) In particular, the two-dimensional sets given by
\begin{align}\eqlab{Ysu}
(v_2,-i\Delta z,\Delta z)\quad \mbox{and}\quad (v_2,i\Delta z,\Delta z),
\end{align} 
define invariant stable and unstable manifolds, respectively, of the line $(v_2,0,0)$, $v_2\in I$, for $\epsilon=0$. For $0\le \epsilon\ll 1$,
we therefore have stable and unstable manifolds (which are line bundles by the linearity):
\begin{align}\eqlab{su2}
    \begin{cases} \Delta y = Y_2^s(v_2,\epsilon) \Delta z,\\
    \Delta y = Y_2^u(v_2,\epsilon) \Delta z,
    \end{cases} 
    \end{align}
    where 
    \begin{align*}
    \begin{cases}
    Y_2^s(v_2,\epsilon)=-i+\mathcal O(\epsilon),\\ Y_2^u(v_2,\epsilon)  =i+\mathcal O(\epsilon).\end{cases}
    \end{align*}
    are $C^n$-smooth (for any $n\in \mathbb N$) with respect to $v_2\in I$, $0\le \epsilon<\epsilon_0(n)$, $0<\epsilon_0(n)\ll 1$, with $I\subset \mathbb R$ a fixed compact set. This follows from Fenichel's theory, see e.g. \cite{jones_1995}. In fact, by the symmetry \eqref{symmetry}, \eqref{Deltayz2} has a time-reversal symmetry $(v_2,\Delta y,\Delta z,t)\mapsto (-v_2,\overline{\Delta y},\overline{\Delta z},-t)$, and we can therefore take
    \begin{align*}
        Y_2^u(v_2,\epsilon) := \overline{Y_2^s}(-v_2,\epsilon).
    \end{align*}
In turn,
\begin{align}\eqlab{diagonalization1}
 \begin{pmatrix}
  \Delta y\\
  \Delta z
 \end{pmatrix}  &=
\begin{pmatrix}
 Y_2^s(v_2,\epsilon) & \overline{Y_2^s}(-v_2,\epsilon)\\
 1 & 1 \end{pmatrix}\begin{pmatrix}
 \Delta s\\
 \Delta u
 \end{pmatrix},
\end{align}
defines a $C^n$-smooth $\epsilon$-dependent change of coordinates $(v_2,\Delta y,\Delta z)\mapsto (v_2,\Delta s,\Delta u)$, that  diagonalizes \eqref{Deltayz2}:
\begin{equation}\eqlab{p2Duv}
\begin{aligned}
v_2' &=-\epsilon,\\
\Delta s'&= \frac{1}{v_2^2+1} \left(-1+\mathcal O(\epsilon)\right)\Delta s,\\
\Delta u'&= \frac{1}{v_2^2+1} \left(1+\mathcal O(\epsilon)\right)\Delta u.
\end{aligned}
\end{equation}
In these coordinates, the time-reversal symmetry takes the form
$$(v_2,\Delta s,\Delta u,t)\mapsto (-v_2,\overline{\Delta u},\overline{\Delta s},-t). $$
The system \eqref{p2Duv} can be viewed as a Fenichel normal form, see \cite[Section 3]{jones_1995}.
 The following result shows that the diagonalization  \eqref{diagonalization1} can be ``globalized'' beyond a compact subset $v_2\in I$ (by working in the directional charts $\breve v=\pm 1$ associated with the blowup \eqref{blowupfinal}).

\begin{proposition}\proplab{diag}
Fix any $n\in \mathbb N$ and consider \eqref{DeltayDeltaz} with $x=iv$, $v\in [-\delta,\delta]$, $\delta>0$ small enough. Then there exists a locally defined $\epsilon$-dependent and $v$-fibered change of coordinates $(v,\Delta y,\Delta z)\mapsto (v,\Delta s,\Delta u)$ defined by
\begin{align*}
    \begin{pmatrix}
        \Delta y\\
        \Delta z
    \end{pmatrix} = \begin{pmatrix}
        -i+\mathcal O(\rho^2) & i +\mathcal O(\rho^2)\\
        1 & 1
    \end{pmatrix} \begin{pmatrix}
        \Delta s\\
        \Delta u
    \end{pmatrix},
\end{align*}
that diagonalizes \eqref{DeltayDeltaz}:
\begin{equation}\eqlab{diagonalfinal}
    \begin{aligned}
        \frac{d\Delta s}{dv} &=\frac{1}{-v^2-\epsilon^2}\left(-1+\frac{\FF i v^3}{v^2+\epsilon^2}+b v+i\epsilon b\sigma  +\mathcal O(\rho^2)\right)\Delta s,\\
        \frac{d\Delta u}{dv} &= \frac{1}{-v^2-\epsilon^2}\left(1-\frac{\FF i v^3}{v^2+\epsilon^2}+b v+i\epsilon b\sigma  +\mathcal O(\rho^2)\right)\Delta u.
    \end{aligned}
    \end{equation}
    Here all $\mathcal O(\rho^2)$-terms are $C^n$-smooth with respect to $\rho \in [0,\delta]$, $(\breve v,\breve \epsilon)\in S^1$ (i.e. they are smooth functions of $(\rho_1,\epsilon_1)$, $(v_2,\rho_2=\epsilon)$ and $(\rho_3,\epsilon_3)$ in the charts). Moreover, the system \eqref{diagonalfinal} is invariant with respect to the time-reversal symmetry $(v,\Delta s,\Delta u)\mapsto (-v,\overline{\Delta u},\overline{\Delta s})$ for all $0<\epsilon\ll 1$.
\end{proposition}
\begin{proof}
    In order to prove the statement, we will extend (globalize) the stable and unstable manifolds of \eqref{Deltayz2}, recall \eqref{su2}, beyond the scaling chart ($\breve \epsilon=1$). For this purpose, it is convenient to work in projective coordinates: For $\Delta z\ne 0$, define
    \begin{align*}
        Y := \frac{\Delta y}{\Delta z}.
    \end{align*}
    Then a simple calculation shows that
    \begin{align*}
        \frac{dY}{dv} = i\left(A_{12}(iv,\epsilon)+(A_{11}(iv,\epsilon)-A_{22}(iv,\epsilon))Y-A_{21}(iv,\epsilon)Y^2\right),
    \end{align*}
    which is invariant with respect to $(v,Y)\mapsto (-v,\overline Y)$, recall \eqref{Eprop}.
    In the $\breve \epsilon=1$-chart, we obtain the first order system
    \begin{equation}\eqlab{p2Y}
    \begin{aligned}
v_2'&=-\epsilon,\\
Y' & = \frac{i}{v_2^2+1+\epsilon \FF i v_2^3} \left(1+\mathcal O(\epsilon^2) + \mathcal O(\epsilon^2) Y + (1+\mathcal O(\epsilon^2))Y^2\right),\\
\epsilon' &=0,
    \end{aligned}
    \end{equation}
    by \eqref{Aalphabeta2}.
    Now, for $\epsilon=0$ we have critical manifolds defined by $(v_2,\mp i,0)$ which are normally hyperbolic. Indeed, the linearization about any point $(v_2,\mp i,0)$ has a single nontrivial eigenvalue of the form
    \begin{align}\eqlab{nontrivialeig}
     \frac{i}{v_2^2+1}2(\mp i) = \frac{\pm 2}{v_2^2+1},
    \end{align}
respectively. We focus on $Y=-i$ (and use the symmetry to study $Y=i$). By \eqref{Ysu}, $Y=-i$ corresponds to the stable manifold in the $(v_2,\Delta y,\Delta z)$-space for $\epsilon=0$ and notice that it is unstable for \eqref{p2Y} by \eqref{nontrivialeig}. Now, by working in a compact set $v_2\in I$, we have a repelling slow manifold:
    \begin{align}\eqlab{slow}
       Y = -i+\mathcal O(\epsilon^2),
    \end{align}
    as a $C^n$-smooth graph over $(v_2,\epsilon)\in I\times [0,\epsilon_0)$, $0<\epsilon_0\ll 1$.
    Next, we consider the coordinates of the $\breve v=1$-chart:
    \begin{align*}
    A_{11}(i\rho_1,\rho_1\epsilon_1) &= \frac{1}{\rho_1^2 \left[-1-\epsilon_1^2 -\FF i \rho_1\right]}\left(b \rho_1(-i+\epsilon_1 \sigma) +\mathcal O(\rho_1^2)\right),\\
    A_{12}(i\rho_1,\rho_1\epsilon_1) &= \frac{1}{\rho_1^2 \left[-1-\epsilon_1^2 -\FF i \rho_1\right]}\left(1+\mathcal O(\rho_1^2)\right),\\
    A_{21}(i\rho_1,\rho_1\epsilon_1) &= \frac{1}{\rho_1^2 \left[-1-\epsilon_1^2 -\FF i \rho_1\right]}\left(-1+\mathcal O(\rho_1^2)\right),\\
     A_{22}(i\rho_1,\rho_1\epsilon_1) &= \frac{1}{\rho_1^2 \left[-1-\epsilon_1^2 -\FF i \rho_1\right]}\left(b \rho_1(-i+\epsilon_1 \sigma) +\mathcal O(\rho_1^2)\right).
\end{align*}
    Consequently, 
    \begin{align*}
    \frac{dY}{d\rho_1} &= \frac{i}{\rho_1^2\left[-1-\epsilon_1^2-\FF i \rho_1\right]} \left(1+\mathcal O(\rho_1^2)+\mathcal O(\rho_1^2)Y+(1+\mathcal O(\rho_1)) Y^2\right),
    \end{align*}
    or equivalently
    \begin{equation}\eqlab{Deltayz1}
\begin{aligned}
    \dot \rho_1 &= -\rho_1^2,\\
    \dot Y &=  \frac{i}{\left[1+\epsilon_1^2+\FF i \rho_1\right]}\left(1+\mathcal O(\rho_1^2)+\mathcal O(\rho_1^2)Y+(1+\mathcal O(\rho_1)) Y^2\right),\\
    \dot \epsilon_1 &= \rho_1 \epsilon_1.
\end{aligned}
\end{equation}
In comparison with \eqref{p2Y}, we have used a desingularization of time corresponding to division of the right hand side by $\epsilon_1^2$, i.e. $\dot{(\cdot)} = \epsilon_1^{-2} (\cdot)'$. For \eqref{Deltayz1}, we have partially hyperbolic singularities $(0,\pm i,0)$, that are stable and unstable, respectively. Indeed the linearization around $(0,\pm i,0)$ have eigenvalues $0,\mp 2,0$. We focus on $(0,-i,0)$, having a  center manifold:
\begin{align}\eqlab{center1}
    Y = -i +\mathcal O(\rho_1^2),
\end{align}
as a $C^n$-smooth graph over $\rho_1\in [0,\delta],\epsilon_1\in [0,\xi]$, with $\delta>0$ and $\xi>0$ small enough. The center manifold, which is repelling for \eqref{Deltayz1}, is foliated by $\epsilon=\rho_1\epsilon_1$ constant.

The results in the $\breve v=-1$-chart are similar. Here we also obtain a partially hyperbolic equilibrium point $(\rho_3,Y,\epsilon_3)=(0,-i,0)$ and a two-dimensional center manifold:
\begin{align}
    Y = -i +\mathcal O(\rho_3^2),\eqlab{center3}
\end{align}
as a $C^n$-smooth graph over $\rho_3\in [0,\delta],\epsilon_3\in [0,\xi]$, with $\delta>0$ and $\xi>0$ small enough. The center manifold is also repelling and foliated by $\epsilon=\rho_3\epsilon_3$ constant.

Now, notice that with respect to \textit{the backward flow}, the center manifolds in the $\breve v=\pm 1$-charts and the slow manifold in the $\breve \epsilon=1$-chart all become attracting and moreover $v$ (in particular $v_2$) becomes an increasing function. Therefore, following the pioneering work of Dumortier and Roussarie \cite{dumortier_1996}, see also \cite{krupa_extending_2001} by Krupa and Szmolyan, we can use the backward flow to extend a fixed center manifold in the $\breve v=-1$-chart, see \eqref{center3}, into the scaling chart $\breve \epsilon=1$ (since $\rho_3=-p$ is decreasing) and continue this as a smooth invariant manifold there upon using the slow manifold, see \eqref{slow}, which is also attracting for the backward flow. Moreover, from here we can extend the resulting manifold into the $\breve v=1$-chart by using the attracting center manifold, see \eqref{center1}. In this way, we obtain a global invariant manifold
\begin{align*}
Y=Y^s(v,\epsilon)=-i+\mathcal O(\rho^2),
\end{align*}
with $\mathcal O(\rho^2)$ being $C^n$-smooth on the cylinder (i.e. in the charts). 
%

Next, we obtain the unstable manifold by applying the time-reversal symmetry:
\begin{align*}
    Y=Y^u(v,\epsilon):=\overline{Y^s}(-v,\epsilon).
    \end{align*}
   Then the desired diagonalization is given by
    \begin{align}\eqlab{Deltayzuv}
 \begin{pmatrix}
  \Delta y\\
  \Delta z
 \end{pmatrix}  = \begin{pmatrix}
 Y^s(v,\epsilon) & \overline{Y^s}(-v,\epsilon)\\
 1 & 1 \end{pmatrix}\begin{pmatrix}
 \Delta s\\
 \Delta u
 \end{pmatrix}
    \end{align}
    which brings the stable and unstable manifolds to $\Delta u = 0$ and $\Delta s=0$, respectively. In turn, it follows that the linear system has been diagonalized, so that $\frac{d\Delta }{dv}$ only depends upon $\Delta $ for $\Delta =\Delta s,\Delta u$.  Therefore to compute ${d\Delta s}/{dv}$ we can just put $\Delta u=0$ so that $\Delta z=\Delta s$ and set $\Delta y=Y^s(v,\epsilon)\Delta z$ (cf. \eqref{Deltayzuv}) into the $\Delta z$-equation:
\begin{align*}
\frac{d\Delta s}{dv} &= i\left(A_{21}(iv,\epsilon) Y^s(v,\epsilon)+A_{22}(iv,\epsilon)\right)\Delta s.
\end{align*} 
${d\Delta u}/{dv}$ is obtained in the same way:
\begin{align*}
    \frac{d\Delta u}{dv} &= i\left(A_{21}(iv,\epsilon) \overline{Y^s}(-v,\epsilon)+A_{22}(iv,\epsilon)\right)\Delta u.
\end{align*}
Now, using the expansions of $A_{21},A_{22}$ and $Y^s$ in the charts, we obtain the desired result. In particular, in the $\breve v=1$-chart,  we find that
\begin{align*}
\frac{d\Delta s}{d\rho_1} &= \frac{1}{\rho_1^2 \left[-1-\epsilon_1^2 -\FF i \rho_1\right]}\left(-1+ib\rho_1(-i+\epsilon_1 \sigma)+\mathcal O(\rho_1^2)\right)\Delta s,\\
&=\frac{1}{\rho_1^2 \left[-1-\epsilon_1^2\right] }\left(-1+\frac{\FF i\rho_1}{1+\epsilon_1^2}+b\rho_1 +i \rho_1 \epsilon_1 b\sigma+\mathcal O(\rho_1^2)\right)\Delta s,
\end{align*} 
using the expansion
\begin{align*}
     \left[1+\epsilon_1^2 +\FF i \rho_1\right]^{-1} = \frac{1}{1+\epsilon_1^2 } -\frac{\FF i\rho_1 }{(1+\epsilon_1^2 )^2}+\mathcal O(\rho_1^2),
\end{align*}
uniformly on a compact domain $\epsilon_1\in [0,\xi]$, $\xi>0$,
and similarly
\begin{align*}
\frac{d\Delta u}{d\rho_1} &=\frac{1}{\rho_1^2 \left[-1-\epsilon_1^2\right] }\left(1-\frac{\FF i\rho_1}{1+\epsilon_1^2}+b\rho_1 +i \rho_1 \epsilon_1 b\sigma+\mathcal O(\rho_1^2)\right)\Delta u.
\end{align*} 

    
\end{proof}

  \begin{figure}[h!]
\begin{center}
{\includegraphics[width=.75\textwidth]{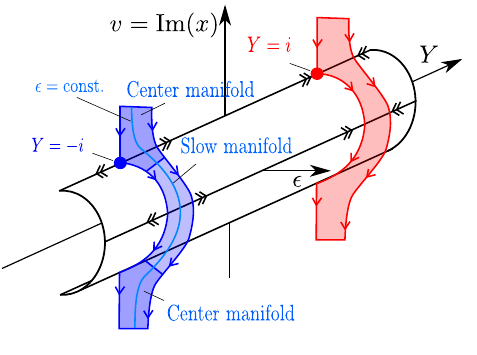}}
\end{center}
\caption{Schematic illustration of the use of blowup to perform the extension of the invariant manifolds used to diagonalize \eqref{DeltayDeltaz}, see \propref{diag}. There are slow manifolds (light blue) on the side of the cylinder (obtained by working in a large but fixed compact sets of the $\breve \epsilon=1$-chart ) and center manifolds (darker blue) on the top and bottom (working in small compact sets of the charts $\breve v=\pm 1$, respectively). In particular, we obtain the global invariant manifold (in blue) by fixing a center manifold on the bottom and then extend this manifold up the cylinder by using the backward flow (where the manifolds become attracting). The invariant manifolds near $Y=-i$ (blue) and $Y=i$ (red) are related by symmetry $(v,Y)\mapsto (-v,\overline Y)$. }
\figlab{center_extension}
\end{figure}

\subsection{Computing the difference}
We will now determine a formula  for the difference $(\Delta y,\Delta z)$ at $x=0$ as $\epsilon\rightarrow 0$ (recall \eqref{asymp}). For this purpose, we will use \eqref{diagonalfinal} and the following result regarding the unperturbed manifolds:
\begin{lemma}
    Suppose that the unperturbed manifolds $(y,z)=\psi^{\pm}_0(x)$, $x\in S^{\pm}$, recall \lemmaref{m0pm}, are non-analytic at the origin (i.e. the series \eqref{gevrey1series} is divergent for any $x\ne 0$). Then
\begin{align} \eqlab{m0nonzero}
\psi^+_0(iv)\ne \psi_0^-(iv)\quad \forall\, v\ne 0.
\end{align}
\end{lemma}
\begin{proof}
The proof is by contradiction. First, we note that
$\psi^+_0$ and $\psi^-_0$ satisfy the same differential equation on $x\in S^+\cap S^-$, recall \lemmaref{m0pm}.
This equation is clearly regular away from $x\ne 0$. Therefore if $\psi^+_0(x)=\psi^-_0(x)$ for some $x\in S^+\cap S^-$, $x\ne 0$, then by uniqueness of solutions, we have $\psi^+_0(x)=\psi^-_0(x)$ for all $x\in B_\xi$, i.e. in a full neighborhood of the origin. But then by Riemann's theorem of removable singularities, we conclude that the unperturbed manifolds are analytic at the origin.

\end{proof}
Now, we fix $n\in \mathbb N$ and $\delta>0$ small enough, and suppose that the unperturbed manifolds are non-analytic at the origin. It then follows that
\begin{align*}
   \Theta(\epsilon):&= 
  \begin{pmatrix}
 Y^s(\delta ,\epsilon) & \overline{Y^s}(-\delta,\epsilon)\\
 1 & 1 \end{pmatrix}^{-1}  \left(m^+(i\delta,\epsilon)-m^-(i\delta,\epsilon)\right),
\end{align*}
is $C^n$-smooth with respect to $\epsilon\in [0,\epsilon_0)$, and importantly:
\begin{align}\eqlab{Theta0}\Theta(0)\ne (0,0)\in \mathbb C^2,\end{align} by \eqref{m1eps0} and \eqref{m0nonzero}. Here $Y^s=Y^s(v,\epsilon)$, $v\in [-\delta,\delta]$, $\epsilon\in [0,\epsilon_0)$, is the $C^n$-smooth function from \propref{diag}.

Let $(\Theta_s,\Theta_u)$ denote the components of $\Theta\in \mathbb C^2$. By \eqref{Theta0}, it follows that either $\Theta_s(0)\ne 0$ or $\Theta_u(0)\ne 0$. Suppose the latter and consider the solution $\Delta s=\Delta s(v)$ of the $\Delta s$-equation:
\begin{align}\eqlab{ueqn}
\frac{d\Delta s}{dv} = \frac{1}{-v^2-\epsilon^2}\left(-1+\frac{\FF i v^3}{v^2+\epsilon^2}+b v+i\epsilon b\sigma  +\mathcal O(\rho^2)\right)\Delta s,
\end{align}
with $\Delta s(-\delta)=\overline{\Theta_u}$ as the initial condition.
Since the system is linear, we can integrate the equation directly and if $\Theta_u(0)\ne 0$ then we easily conclude that $\Delta s(0)$ is exponentially large with respect to $\epsilon\rightarrow 0$ (see also \secref{completing} below). But this contradicts \propref{unstable2} and the boundedness of $m_2^{\pm}(0,\epsilon)$. We therefore conclude that $\Theta_s(0)\ne 0$ and instead use $\Delta s(\delta)=\Theta_s$ as an initial condition for \eqref{ueqn}.
We obtain 
\begin{align}\eqlab{Deltauexp}
    \Delta s (0) = \e^{\int_{\delta}^0 \frac{1}{-v^2-\epsilon^2}\left(-1+\frac{\FF i v^3}{v^2+\epsilon^2}+b v+i\epsilon b\sigma  +\mathcal O(\rho^2)\right) dv } \Theta_s.
\end{align}
To determine $\Delta u(0)$ we use the symmetry: $\Delta u(v)=\overline{\Delta s}(-v)$ for all $v\in [-\delta,0]$. In particular, $\Delta u(0)=\overline {\Delta s}(0)$. We are now ready to complete the proof of the main result, \thmref{main}.

\subsection{Completing the proof of \thmref{main}}\seclab{completing}
  We start by analyzing the integral in   \eqref{Deltauexp}:
\begin{align*}
    \int_{\delta}^0 \frac{1}{-v^2-\epsilon^2}\left(-1+\frac{\FF i v^3}{v^2+\epsilon^2}+b v+i\epsilon b\sigma  +\mathcal O(\rho^2)\right) dv.
\end{align*}
We expand the integrand and consider the separate terms successively. First, we directly obtain that
\begin{align}
    \int_{\delta}^0 \frac{1}{v^2+\epsilon^2}dv = -\epsilon^{-1} \arctan \frac{\delta}{\epsilon} = -\frac{\pi}{2}\epsilon^{-1} +\mathcal O(1),\eqlab{int0}
\end{align}
since 
\begin{align*}
    \arctan s= \frac{\pi}{2}+\mathcal O(s^{-1})\quad \mbox{for $s\rightarrow \infty$};
\end{align*}
  recall that $\delta>0$ is fixed. Next, we similarly have  that
  \begin{align*}
      \int_{\delta}^0 \frac{-\FF i v^3}{(v^2+\epsilon^2)^2}dv &=-\frac12 \FF i \left[ \log (v^2+\epsilon^2)+\frac{\epsilon^2}{\epsilon^2+p^2} \right]_{v=\delta}^0\\
      &=-\frac12 \FF i \left(\log \frac{\epsilon^2}{\delta^2+\epsilon^2}+1-\frac{\epsilon^2}{\epsilon^2+\delta^2}\right)\\
      &= \FF i  \log \epsilon^{-1} +\mathcal O(1),
  \end{align*}
  and 
  \begin{align}
    \int_{\delta}^0 \frac{-bv}{v^2+\epsilon^2}dv &=-\frac12 b\left[ \log \frac{\epsilon^2}{v^2+\epsilon^2}\right]_{v=\delta}^0\\
    &=b \log \epsilon^{-1}  +\mathcal O(1),\nonumber\\
    \int_{\delta}^0 \frac{-i \epsilon b\sigma }{v^2+\epsilon^2}dv&=\frac{i \pi b \sigma}{2}+\mathcal O(\epsilon),\eqlab{pis}
\end{align}
using \eqref{int0} in the last integral. It remains to consider the integral
\begin{align*}
    \int_{\delta}^0 \frac{\mathcal O(\rho^2)}{-v^2-\epsilon^2} dv,
\end{align*}
which we analyze by working in the coordinates $(\rho_1,\epsilon_1)$ and $(v_2,\epsilon)$ associated with the blowup \eqref{blowupfinal}. In particular, we split the integral
\begin{align*}
    \int_{\delta}^0  \frac{\mathcal O(\rho^2)}{-v^2-\epsilon^2} dv &=\left(\int_{\delta}^{\xi^{-1}\epsilon}+\int_{\xi^{-1}\epsilon}^0\right)  \frac{\mathcal O(\rho^2)}{-v^2-\epsilon^2} dv \\
    &=\int_{\delta}^{\xi^{-1} \epsilon}  \frac{\mathcal O(1)}{1+(\epsilon \rho_1^{-1})^2} d\rho_1+\int_{\xi^{-1}}^0 \frac{\mathcal O(\epsilon )}{-v_2^2-1} dv_2\\
    &=\mathcal O(1),
\end{align*}
with $\xi>0$ fixed small enough,
using the substitutions $v=\rho_1$ and $v=\epsilon v_2$ in the second equality. Here we have also used that under these substitutions $\mathcal O(\rho^2)$ becomes $\mathcal O(\rho_1^2)$ and $\mathcal O(\epsilon^2)$, respectively. It is easy to see that all $\mathcal O(1)$ terms are $C^n$-smooth with respect to $\epsilon$. In conclusion, we have
\begin{align*}
    \Delta s (0) &= \e^{-\frac{\pi}{2}\epsilon^{-1} +\FF i  \log \epsilon^{-1}  +b \log \epsilon^{-1}  +\mathcal O(1)}\\
    &=\epsilon^{-b}\e^{-\frac{\pi}{2}\epsilon^{-1}+\FF i  \log \epsilon^{-1}+\mathcal O(1)},
\end{align*}
where $\mathcal O(1)\in \mathbb C$ is $C^n$-smooth with respect to $\epsilon$. Here we have used that $\Theta_s(0)\ne 0$. Then by \eqref{Deltayzuv}, we have
       \begin{align*}
 \begin{pmatrix}
  \Delta y\\
  \Delta z
 \end{pmatrix}(0)  &= \begin{pmatrix}
 Y^s(0,\epsilon) & \overline{Y^s}(-0,\epsilon)\\
 1 & 1 \end{pmatrix}\begin{pmatrix}
 \Delta s(0)\\
 \overline{\Delta s}(0)
 \end{pmatrix}\\
 &=2\operatorname{Re}\left(\begin{pmatrix}
 Y^s(0,\epsilon) \\
 1 \end{pmatrix}
 \Delta s(0)\right),
 \end{align*}
 with $Y^s(0,\epsilon) = -i+\mathcal O(\epsilon)$, recall \eqref{slow}.
 A simple computation then shows \eqref{asymp}. Notice, from the construction above, that the functions $\phi_\alpha$, $\alpha\in \{0,1\}$, in \eqref{asymp} are apriori only $C^n$.  However, the manifolds are unique and smooth for $\epsilon>0$ (in particular, they are independent of our choice of $n$ and $\delta>0$). Therefore by the arbitrariness of $n\in \mathbb N$, we conclude that $\phi_\alpha$, $\alpha\in \{0,1\}$, are indeed $C^\infty$ with respect $\epsilon\in [0,\epsilon_0)$.

\subsection*{Acknowledgement}
The author is grateful to Inmaculada Baldom\'a and Tere M. Seara at the Polytechnic University of Catalonia, Barcelona, for the hospitality shown during two research visits in late 2024 and summer 2025. I am thankful for their kindness, interest and guidance. The author would also like to thank his close collaborators, Sam Jelbart and Peter Szmolyan, for the discussions we had during the early stages of this project in the early summer of 2024. Finally, the author thanks Sam Jelbart and Johan M. Christensen for feedback on an earlier version. The author was funded by Danish Research Council (DFF) grant 4283-00014B.

\newpage 
\bibliography{refs}

@article {ksum,
    AUTHOR = {De Maesschalck, P. and Kristiansen, K.~U.},
     TITLE = {On {$k$}-summable normal forms of vector fields with one zero
              eigenvalue},
   JOURNAL = {Qual. Theory Dyn. Syst.},
  FJOURNAL = {Qualitative Theory of Dynamical Systems},
    VOLUME = {25},
      YEAR = {2026},
    NUMBER = {1},
     PAGES = {4, 33},
      ISSN = {1575-5460,1662-3592},
   MRCLASS = {34C23 (34C45 37G05 37G10)},
  MRNUMBER = {5001492},
       DOI = {10.1007/s12346-025-01429-1},
       URL = {https://doi-org.proxy.findit.cvt.dk/10.1007/s12346-025-01429-1},
}

@article{constantine1996a,
  author = {Constantine, G. M. and Savits, T. H.},
  title = "{A multivariate Faa di Bruno formula with applications}",
  language = {eng},
  format = {article},
  journal = {Transactions of the American Mathematical Society},
  volume = {348},
  number = {2},
  pages = {503-520},
  year = {1996},
  issn = {00029947, 10886850},
  publisher = {American Mathematical Society},
  doi = {10.1090/s0002-9947-96-01501-2}
}

@article {MR4455359,
    AUTHOR = {Baldom\'a, I. and Giralt, M. and Guardia, M.},
     TITLE = {Breakdown of homoclinic orbits to {$L_3$} in the {RPC}3{BP}({I}). {C}omplex singularities and the inner equation},
   JOURNAL = {Adv. Math.},
  FJOURNAL = {Advances in Mathematics},
    VOLUME = {408},
      YEAR = {2022},
     PAGES = {Paper No. 108562, 64},
      ISSN = {0001-8708,1090-2082},
   MRCLASS = {37J46 (70F07 70F15)},
  MRNUMBER = {4455359},
MRREVIEWER = {Mingliang\ Song},
       DOI = {10.1016/j.aim.2022.108562},
       URL = {https://doi-org.proxy.findit.cvt.dk/10.1016/j.aim.2022.108562},
}

@article {MR4621957,
    AUTHOR = {Baldom\'a, I. and Giralt, M. and Guardia, M.},
     TITLE = {Breakdown of homoclinic orbits to {$L_3$} in the {RPC}3{BP}({II}). {A}n asymptotic formula},
   JOURNAL = {Adv. Math.},
  FJOURNAL = {Advances in Mathematics},
    VOLUME = {430},
      YEAR = {2023},
     PAGES = {Paper No. 109218, 72},
      ISSN = {0001-8708,1090-2082},
   MRCLASS = {37J46 (70F07 70F15)},
  MRNUMBER = {4621957},
MRREVIEWER = {Mingliang\ Song},
       DOI = {10.1016/j.aim.2023.109218},
       URL = {https://doi-org.proxy.findit.cvt.dk/10.1016/j.aim.2023.109218},
}

@article {MR4445442,
    AUTHOR = {Merle, F. and Rapha\"{e}l, P. and Rodnianski, I. and
              Szeftel, J.},
     TITLE = "{On the implosion of a compressible fluid {I}: {S}mooth
              self-similar inviscid profiles}",
   JOURNAL = {Ann. of Math. (2)},
  FJOURNAL = {Ann. of Math.. Second Series},
    VOLUME = {196},
      YEAR = {2022},
    NUMBER = {2},
     PAGES = {567--778},
      ISSN = {0003-486X,1939-8980},
   MRCLASS = {35Q35 (34C37)},
  MRNUMBER = {4445442},
MRREVIEWER = {Wei\ Lian},
       DOI = {10.4007/annals.2022.196.2.3},
       URL = {https://doi.org/10.4007/annals.2022.196.2.3},
}

@article{lazutkin2005a,
  author = {Lazutkin, V. F.},
  title = {Splitting of separatrices for the Chirikov standard map},
  language = {eng},
  format = {article},
  journal = {Journal of Mathematical Sciences},
  volume = {128},
  number = {2},
  pages = {2687-2705},
  year = {2005},
  issn = {15738795, 10723374},
  publisher = {Kluwer Academic Publishers-Consultants Bureau},
  doi = {10.1007/s10958-005-0219-7}
}

@article{gelfreich1999a,
  author = {Gelfreich, V. G.},
  title = {A proof of the exponentially small transversality of the separatrices for the standard map},
  language = {eng},
  format = {article},
  journal = {Communications in Mathematical Physics},
  volume = {201},
  number = {1},
  pages = {155-216},
  year = {1999},
  issn = {14320916, 00103616},
  publisher = {Springer New York},
  doi = {10.1007/s002200050553}
}

@article{gaivao2011a,
  author = {Gaiv{\~a}o, J.~P. and Gelfreich, V.},
  title = {Splitting of separatrices for the Hamiltonian-Hopf bifurcation with the Swift-Hohenberg equation as an example},
  language = {eng},
  format = {article},
  journal = {Nonlinearity},
  volume = {24},
  number = {3},
  pages = {677-698},
  year = {2011},
  issn = {13616544, 09517715},
  publisher = {IOP Pub.},
  doi = {10.1088/0951-7715/24/3/002}
}

@article{bonckaert2008a,
  author = {Bonckaert, P. and De Maesschalck, P.},
  title = {Gevrey normal forms of vector fields with one zero eigenvalue},
  language = {eng},
  format = {article},
  journal = {Journal of Mathematical Analysis and Applications},
  volume = {344},
  number = {1},
  pages = {301-321},
  year = {2008},
  issn = {10960813, 0022247x},
  publisher = {ACADEMIC PRESS INC ELSEVIER SCIENCE},
  doi = {10.1016/j.jmaa.2008.02.060}
}

@article{new,
  author = {Kristiansen, K.~U.},
  title = "{A geometric approach to exponentially small splitting: Zero-Hopf bifurcations of arbitrary co-dimension}",
  language = {eng},
  format = {article},
  journal = {arXiv-preprint::2603.12115},
    year = {2026}
}

@book{coppel1978a,
  author = {Coppel, W.~A.},
  title = {Dichotomies in Stability Theory},
  language = {eng},
  format = {book},
  journal = {Dichotomies in Stability Theory},
  volume = {629},
  pages = {Online-Ressource (unknown)},
  year = {1978},
  isbn = {354008536X, 354008536x, 3540359761, 9783540085362, 9783540359760, 038708536x, 9780387085364},
  publisher = {Springer}
}

@article{martinet1983a,
  author = {Martinet, J. and Ramis, J.-P.},
  title = {Classification analytique des \'equations diff\'erentielles non lin\'eaires r\'esonnantes du premier ordre},
  language = {und},
  format = {article},
  journal = {Annales Scientifiques De L'\'ecole Normale Sup\'erieure},
  volume = {16},
  number = {4},
  pages = {571-621},
  year = {1983},
  issn = {18732151, 00129593},
  doi = {10.24033/asens.1462}
}

@article {sauzin2015,
    AUTHOR = {Sauzin, D.},
     TITLE = {Nonlinear analysis with resurgent functions},
   JOURNAL = {Ann. Sci. \'{E}c. Norm. Sup\'{e}r. (4)},
  FJOURNAL = {Annales Scientifiques de l'\'{E}cole Normale Sup\'{e}rieure. Quatri\`eme
              S\'{e}rie},
    VOLUME = {48},
      YEAR = {2015},
    NUMBER = {3},
     PAGES = {667--702},
      ISSN = {0012-9593},
   MRCLASS = {30B40},
  MRNUMBER = {3377056},
MRREVIEWER = {Marek Jarnicki},
       DOI = {10.24033/asens.2255},
       URL = {https://doi.org/10.24033/asens.2255},
}

@book {costin2009,
    AUTHOR = {Costin, O.},
     TITLE = {Asymptotics and {B}orel summability},
    SERIES = {Chapman \& Hall/CRC Monographs and Surveys in Pure and Applied
              Mathematics},
    VOLUME = {141},
 PUBLISHER = {CRC Press, Boca Raton, FL},
      YEAR = {2009},
     PAGES = {xiv+250},
      ISBN = {978-1-4200-7031-6},
   MRCLASS = {34-02 (34M30 34M37 34M40 34M55 39A11 40G10)},
  MRNUMBER = {2474083},
MRREVIEWER = {B. L. J. Braaksma},
}

@article{kristiansen2025a,
  author = {Kristiansen, K.~U.},
  title = {Improved Gevrey-1 estimates of formal series expansions of center manifolds},
  language = {eng},
  format = {article},
  journal = {Studies in Applied Mathematics},
  volume = {154},
  number = {6},
  pages = {e70063},
  year = {2025},
  issn = {14679590, 00222526},
  publisher = {Wiley},
  doi = {10.1111/sapm.70063}
}

@article{krupa_extending_2001,
	Author = {Krupa, M. and Szmolyan, P.},
	Journal = {{SIAM} Journal on Mathematical Analysis},
	Number = {2},
	Pages = {286--314},
	Title = {Extending geometric singular perturbation theory to nonhyperbolic points - Fold and canard points in two dimensions},
	Url = {http://epubs.siam.org/doi/abs/10.1137/S0036141099360919},
	Urldate = {2014-06-02},
	Volume = {33},
	Year = {2001}}

@article{kruskal1991a,
  author = {Kruskal, M.~D. and Segur, H.},
  title = {ASYMPTOTICS BEYOND ALL ORDERS IN A MODEL OF CRYSTAL-GROWTH},
  language = {eng},
  format = {article},
  journal = {Studies in Applied Mathematics},
  volume = {85},
  number = {2},
  pages = {129-181},
  year = {1991},
  issn = {00222526, 14679590},
  publisher = {BLACKWELL PUBLISHERS},
  doi = {10.1002/sapm1991852129}
}

@book{meiss2007a,
  author = {Meiss, J.~ D.},
  title = {Differential dynamical systems},
  language = {eng},
  format = {book},
  volume = {14},
  pages = {XXI, 412 S. (unknown)},
  year = {2007},
  isbn = {0898716357, 0898718236, 9780898716351, 9780898718232},
  publisher = {Society for Industrial and Applied Mathematics}
}

@book{dumortier2006a,
  title = {Qualitative theory of planar differential systems},
  language = {eng},
  publisher = {Springer Berlin Heidelberg},
  journal = {Qualitative Theory of Planar Differential Systems},
  pages = {1-302},
  year = {2006},
  isbn = {3540329021, 3540328939, 9783540329021, 9783540328933},
  doi = {10.1007/978-3-540-32902-2},
  author = {Dumortier, F. and Llibre, J. and Artes, J. C.}
}

@article{chapman1998a,
  author = {Chapman, S. J. and King, J. R. and Adams, K. L.},
  title = {Exponential asymptotics and Stokes lines in nonlinear ordinary differential equations},
  language = {eng},
  format = {article},
  journal = {Proceedings of the Royal Society A: Mathematical, Physical and Engineering Sciences},
  volume = {454},
  number = {1978},
  pages = {2733-2755},
  year = {1998},
  issn = {14712946, 13645021},
  publisher = {Royal Society},
  doi = {10.1098/rspa.1998.0278}
}

@article {MR4814277,
    AUTHOR = {Jelbart, S. and Kuehn, C.},
     TITLE = {Extending discrete geometric singular perturbation theory to
              non-hyperbolic points},
   JOURNAL = {Nonlinearity},
  FJOURNAL = {Nonlinearity},
    VOLUME = {37},
      YEAR = {2024},
    NUMBER = {10},
     PAGES = {105006, 49},
      ISSN = {0951-7715,1361-6544},
   MRCLASS = {37C05 (34D15 37C10 37G10 39A05)},
  MRNUMBER = {4814277},
       DOI = {10.1088/1361-6544/ad72c5},
       URL = {https://doi-org.proxy.findit.cvt.dk/10.1088/1361-6544/ad72c5},
}

@article{dumortier_1996,
	Author = {Dumortier, F. and Roussarie, R.},
	Issue = {557},
	Journal = {Mem. Amer. Math. Soc.},
	Pages = {1-96},
	Title = {Canard cycles and center manifolds},
	Volume = {121},
	Year = {1996}}

@article{uldall2024a,
  author = {Kristiansen, K. U. and Szmolyan, P.},
  title = "{A dynamical systems approach to WKB-methods: The simple turning point}",
  language = {eng},
  format = {article},
  journal = {Journal of Differential Equations},
  volume = {406},
  pages = {202-254},
  year = {2024},
  issn = {10902732, 00220396},
  doi = {10.1016/j.jde.2024.06.006}
}

@article{takens1973a,
  author = {Takens, F.},
  title = {NORMAL FORMS FOR CERTAIN SINGULARITIES OF VECTORFIELDS},
  language = {eng},
  format = {article},
  journal = {Annales De L Institut Fourier},
  volume = {23},
  number = {2},
  pages = {163-195},
  year = {1973},
  issn = {17775310, 03730956},
  publisher = {ANNALES DE L INSTITUT FOURIER},
  doi = {10.5802/aif.467}
}

@article{glendinning2022a,
  author = {Glendinning, P. A. and Simpson, D. J.W.},
  title = {Normal forms for saddle-node bifurcations: Takens' coefficient and applications in climate models},
  language = {eng},
  format = {article},
  journal = {Proceedings of the Royal Society A: Mathematical, Physical and Engineering Sciences},
  volume = {478},
  number = {2267},
  pages = {20220548},
  year = {2022},
  issn = {14712946, 13645021},
  publisher = {Royal Society Publishing},
  doi = {10.1098/rspa.2022.0548}
}

@article{fiedler2025a,
  author = {Fiedler, B.},
  title = {Scalar Polynomial Vector Fields in Real and Complex Time},
  language = {eng},
  format = {article},
  journal = {Regular and Chaotic Dynamics},
  volume = {30},
  number = {2},
  pages = {188-225},
  year = {2025},
  issn = {14684845, 15603547},
  publisher = {Hybrid},
  doi = {10.1134/S1560354725020030}
}

@book{Guckenheimer97,
	Author = {J. Guckenheimer and P. Holmes},
	Edition = {5th},
	Publisher = {Springer Verlag},
	Title = {Nonlinear Oscillations, Dynamical Systems and Bifurcations of Vector Fields},
	Year = {1997}}

@article{baldom2013a,
  author = {Baldom\'a, I. and Castej\'on, O. and Seara, T. M.},
  title = {Exponentially Small Heteroclinic Breakdown in the Generic Hopf-Zero Singularity},
  language = {eng},
  format = {article},
  journal = {Journal of Dynamics and Differential Equations},
  volume = {25},
  number = {2},
  pages = {335-392},
  year = {2013},
  issn = {10407294, 15729222},
  publisher = {SPRINGER},
  doi = {10.1007/s10884-013-9297-2}
}

@article{berry1988a,
  author = {Berry, M.~V.},
  title = {STOKES PHENOMENON - SMOOTHING A VICTORIAN DISCONTINUITY},
  language = {eng},
  format = {article},
  journal = {Publications Mathematiques},
  volume = {68},
  number = {68},
  pages = {211-221},
  year = {1988},
  issn = {16181913, 00738301},
  publisher = {Springer International Publishing},
  doi = {10.1007/BF02698550}
}

@article{6a948bd44bf4469e875c1f1482ef7619,
title = "Exponential asymptotics for elastic and elastic-gravity waves on flow past submerged obstacles",
author = "Lustri, {C.~J.}",
year = "2022",
month = nov,
day = "10",
doi = "10.1017/jfm.2022.806",
language = "English",
volume = "950",
pages = "A6--1--A6--39",
journal = "Journal of Fluid Mechanics",
issn = "0022-1120",
publisher = "Cambridge University Press (CUP)",
}

@article {MR4743478,
    AUTHOR = {Baldom\'a, I. and Seara, T.~M. and Moreno, R.},
     TITLE = {Splitting of separatrices for rapid degenerate perturbations
              of the classical pendulum},
   JOURNAL = {SIAM J. Appl. Dyn. Syst.},
  FJOURNAL = {SIAM Journal on Applied Dynamical Systems},
    VOLUME = {23},
      YEAR = {2024},
    NUMBER = {2},
     PAGES = {1159--1198},
      ISSN = {1536-0040},
   MRCLASS = {37J40 (37D10)},
  MRNUMBER = {4743478},
MRREVIEWER = {Zhengdong\ Du},
       DOI = {10.1137/23M1550992},
       URL = {https://doi-org.proxy.findit.cvt.dk/10.1137/23M1550992},
}

@article {MR4940205,
    AUTHOR = {Baldom\'a, I. and Guardia, M. and Pelinovsky,
              D.~E.},
     TITLE = {On a countable sequence of homoclinic orbits arising near a
              saddle-center point},
   JOURNAL = {Comm. Math. Phys.},
  FJOURNAL = {Communications in Mathematical Physics},
    VOLUME = {406},
      YEAR = {2025},
    NUMBER = {9},
     PAGES = {215, 65},
      ISSN = {0010-3616,1432-0916},
   MRCLASS = {37C29 (35B25 37G20)},
  MRNUMBER = {4940205},
       DOI = {10.1007/s00220-025-05381-8},
       URL = {https://doi-org.proxy.findit.cvt.dk/10.1007/s00220-025-05381-8},
}

@article {MR4892796,
    AUTHOR = {Gomide, O.~M.~L. and Guardia, M. and Seara, T.~M.
              and Zeng, C.},
     TITLE = {On small breathers of nonlinear {K}lein-{G}ordon equations via
              exponentially small homoclinic splitting},
   JOURNAL = {Invent. Math.},
  FJOURNAL = {Inventiones Mathematicae},
    VOLUME = {240},
      YEAR = {2025},
    NUMBER = {2},
     PAGES = {661--777},
      ISSN = {0020-9910,1432-1297},
   MRCLASS = {35L71 (35B10)},
  MRNUMBER = {4892796},
MRREVIEWER = {Satyanad\ Kichenassamy},
       DOI = {10.1007/s00222-025-01327-y},
       URL = {https://doi-org.proxy.findit.cvt.dk/10.1007/s00222-025-01327-y},
}

@article {MR779710,
    AUTHOR = {Broer, H. W. and Vegter, G.},
     TITLE = "{Subordinate Silnikov bifurcations near some
              singularities of vector fields having low codimension}",
   JOURNAL = {Ergodic Theory Dynam. Systems},
  FJOURNAL = {Ergodic Theory and Dynamical Systems},
    VOLUME = {4},
      YEAR = {1984},
    NUMBER = {4},
     PAGES = {509--525},
      ISSN = {0143-3857,1469-4417},
   MRCLASS = {58F14 (58F13)},
  MRNUMBER = {779710},
MRREVIEWER = {C.\ Tresser},
       DOI = {10.1017/S0143385700002613},
       URL = {https://doi-org.proxy.findit.cvt.dk/10.1017/S0143385700002613},
}

@article{baldom2019a,
  author = {Baldom\'a, I. and Ib\'a\~nez, S. and Seara, T.M.},
  title = "{Hopf-zero singularities truly unfold chaos}",
  language = {eng},
  format = {article},
  journal = {Communications in Nonlinear Science and Numerical Simulation},
  volume = {84},
  pages = {105162},
  year = {2019},
  issn = {18787274, 10075704},
  publisher = {Elsevier B.V.},
  doi = {https://doi.org/10.1016/j.cnsns.2019.105162}
}

@article {MR4855745,
    AUTHOR = {Kristiansen, K.~U. and Szmolyan, P.},
     TITLE = {Analytic weak-stable manifolds in unfoldings of saddle-nodes},
   JOURNAL = {Nonlinearity},
  FJOURNAL = {Nonlinearity},
    VOLUME = {38},
      YEAR = {2025},
    NUMBER = {2},
     PAGES = {025019, 70},
      ISSN = {0951-7715,1361-6544},
   MRCLASS = {34C45 (34D35 37G05 37G10)},
  MRNUMBER = {4855745},
MRREVIEWER = {Patrick\ Bonckaert},
       DOI = {10.1088/1361-6544/ada67a},
       URL = {https://doi-org.proxy.findit.cvt.dk/10.1088/1361-6544/ada67a},
}

@article{CHAPMAN2009319,
title = {Exponential asymptotics of localised patterns and snaking bifurcation diagrams},
journal = {Physica D: Nonlinear Phenomena},
volume = {238},
number = {3},
pages = {319-354},
year = {2009},
issn = {0167-2789},
doi = {https://doi.org/10.1016/j.physd.2008.10.005},
url = {https://www.sciencedirect.com/science/article/pii/S0167278908003679},
author = {Chapman, S.~J. and Kozyreff, G.}
}

@article{neishtadt1987a,
  author = {Neishtadt, A.~I.},
  title = {PERSISTENCE OF STABILITY LOSS FOR DYNAMICAL BIFURCATIONS .1},
  language = {eng},
  format = {article},
  journal = {Differential Equations},
  volume = {23},
  number = {12},
  pages = {1385-1391},
  year = {1987},
  issn = {16083083, 00122661},
  publisher = {PLENUM PUBL CORP}
}

@article{neishtadt1988a,
  author = {Neishtadt, A.~I.},
  title = {PERSISTENCE OF STABILITY LOSS FOR DYNAMICAL BIFURCATIONS .2},
  language = {eng},
  format = {article},
  journal = {Differential Equations},
  volume = {24},
  number = {2},
  pages = {171-176},
  year = {1988},
  issn = {00122661},
  publisher = {PLENUM PUBL CORP}
}

@article{rousseau2005a,
  author = {Rousseau, C.},
  title = {MODULUS OF ORBITAL ANALYTIC CLASSIFICATION FOR A FAMILY UNFOLDING A SADDLE-NODE},
  language = {eng},
  format = {article},
  journal = {Moscow Mathematical Journal},
  volume = {5},
  number = {1},
  pages = {245-268},
  year = {2005},
  issn = {16094514, 16093321},
  publisher = {INDEPENDENT UNIV MOSCOW-IUM},
  doi = {10.17323/1609-4514-2005-5-1-245-268}
}

@article{hayes2016a,
  title = {Geometric desingularization of degenerate singularities in the presence of fast rotation: A new proof of known results for slow passage through Hopf bifurcations},
  language = {eng},
  publisher = {Elsevier B.V.},
  journal = {Indagationes Mathematicae},
  volume = {27},
  number = {5},
  pages = {1184-1203},
  year = {2016},
  issn = {18726100, 00193577},
  doi = {10.1016/j.indag.2015.11.005},
  author = {Hayes, M. G. and Kaper, T. J. and Szmolyan, P. and Wechselberger, M.}
}

@book{jones_1995,
	Author = {Jones, C. K. R. T. },
	Publisher = {Springer, Berlin},
	Title = {Geometric Singular Perturbation Theory, Lecture Notes in Mathematics, Dynamical Systems (Montecatini Terme)},
	Year = {1995}}

@article{fen3,
	Author = {Fenichel, N.},
	Journal = {J. Diff. Eq.},
	Pages = {53--98},
	Title = {Geometric singular perturbation theory for ordinary differential equations},
	Volume = {31},
	Year = {1979}}

@article{de2020a,
  title = {Gevrey asymptotic properties of slow manifolds},
  language = {eng},
  publisher = {IOP PUBLISHING LTD},
  journal = {Nonlinearity},
  volume = {33},
  number = {1},
  pages = {341-387},
  year = {2020},
  issn = {13616544, 09517715},
  doi = {10.1088/1361-6544/ab4d86},
  author = {De Maesschalck, P. and Kenens, K.}
}
\bibliographystyle{plain}

\newpage 

\appendix 
 \section{Derivation of the system \eqref{model0}}\applab{model}
 In \cite{baldom2013a}, the authors present the following co-dimension two unfolding of the real-analytic zero-Hopf bifurcation:
\begin{equation}\eqlab{modelapp0}
\begin{aligned}
 x' &=x^2-\mu +\gamma_3 \mu^2 + \gamma_4 \nu^2 + \gamma_5 \mu \nu+\gamma_2 (y^2 +z^2) + F(x,y,z,\mu,\nu),\\
 y'&= (\nu-\beta_1 x ) y +  z(\alpha_0+\alpha_1\nu + \alpha_2 \mu+\alpha_3 x)+G(x,y,z,\mu,\nu),\\
 z'&= (\nu-\beta_1 x) z -  y(\alpha_0+\alpha_1\nu + \alpha_2 \mu+\alpha_3 x)+H(x,y,z,\mu,\nu),
\end{aligned}
\end{equation}
 see \cite[Eq. (6)]{baldom2013a} with $(\overline x,\overline y,\overline z)=:(y,z,x)$, where $\alpha_0\ne 0$. Here $\mu$ and $\nu$ are the local unfolding parameters, with $(\mu,\nu)=(0,0)$ corresponding to the zero-Hopf bifurcation, the linearization having eigenvalues $0$, $\pm i\alpha_0^2$, and $F$, $G$, $H$ are each real-analytic and cubic with respect $(x,y,z,\mu,\nu)$.
  Since $\alpha_0\ne 0$, we can divide the right hand side by
 \begin{align*}
  \alpha_0 +\alpha_1 \nu +\alpha_2 \mu+\alpha_3 x\ne 0
 \end{align*}
for all $x,\mu,\nu$ sufficiently close to $0$.
This corresponds to a regular transformation of time.
%
Then through elementary transformations of $(x,\mu,\nu)$, we can bring the system into the following form
\begin{equation}\eqlab{modelapp1}
\begin{aligned}
 x' &=x^2-\mu+\gamma_2 (y^2 +z^2) + F(x,y,z,\mu,\nu),\\
 y'&= (\nu-\beta_1 x ) y +  z+G(x,y,z,\mu,\nu),\\
 z'&= (\nu-\beta_1 x) z -  y+H(x,y,z,\mu,\nu),
\end{aligned}
\end{equation}
for some new real-analytic functions $F$, $G$ and $H$ and new parameters $\beta_1$ and $\gamma_2$.
 Finally, we notice that the auxiliary equation
\begin{align*}
 x' &=x^2 -\mu + F(x,0,0,\mu,\nu),
\end{align*}
obtained by setting $y=z=0$ in the $x$-equation, is an equation for an analytic unfolding of a planar saddle-node. Here it is well-known that the $x^3$-term of $F(x,0,0,0,0)$ cannot be removed by normal form transformations, see e.g. \cite{takens1973a}. We consider the expansion 
\begin{align*}
 F(x,0,0,\mu,\nu) = \sum_{\alpha+\beta+\gamma=3} F_{\alpha,\beta,\gamma} x^\alpha \mu^\beta \nu^\gamma + \mathcal O(\vert (x,\mu,\nu)\vert^4),
\end{align*}
with $\mathcal O(\vert (x,\mu,\nu)\vert^4)$ denoting terms that are fourth order with respect to $(x,\mu,\nu)\rightarrow (0,0,0)$.
It is then a simple calculation to show that
\begin{align}\eqlab{tildex}
 x = \widetilde x -\frac12 F_{1,2,0} \mu^2 -\frac12 F_{1,0,2} \nu^2-\frac{1}{2}F_{1,1,1} \mu \nu- F_{2,1,0}\widetilde x \mu-F_{2,0,1} \widetilde x \nu 
\end{align}
which defines a local $(\mu,\nu)$-dependent change of coordinates $\widetilde x\mapsto x$, brings the auxiliary equation into the following form
\begin{align*}
 \widetilde x ' = \widetilde x^2 -\widetilde \mu(\mu,\nu) +F_{3,0,0} x^3+\mathcal O(\vert (x,\mu,\nu)\vert^4),
\end{align*}
with $\widetilde \mu(\mu,\nu) = \mu + \mathcal O(2)$ an analytic function. Therefore by applying the $(\mu,\nu)$-dependent change of coordinates  $\widetilde x\mapsto x$ defined by \eqref{tildex} to the full system \eqref{modelapp1},  and a subsequent diffeomorphic change of parameters $(\mu,\nu)\mapsto (\widetilde \mu,\nu)$, we again obtain \eqref{modelapp1} upon dropping the tildes, but where the new $F$ now satisfies
\begin{align}
 F(x,0,0,\mu,\nu) = F_{3,0,0} x^3+\mathcal O(\vert (x,\mu,\nu)\vert^4).\eqlab{Fexpansionapp}
\end{align}
To complete the derivation of \eqref{model0} satisfying \eqref{Fexpansion}, we simply set
%
%
%
\begin{align*}
 \begin{cases} \gamma_2 = a,\\
 \beta_1 = b,\\
 F_{3,0,0}=\FF.
 \end{cases}
\end{align*}
\begin{remark}
 It is an interesting problem whether $\mathcal O(\vert (x,\mu,\nu)\vert^4)$ in \eqref{Fexpansionapp} can be removed completely within the analytic case. The recent paper \cite{glendinning2022a} studies a similar problem in the smooth topology and seem to obtain partial results of this kind.
\end{remark}


\end{document}